\documentclass[a4paper, 11pt,reqno]{amsart}
\usepackage[top = 1in, bottom = 1in, left=1.25in, right=1.25in]{geometry}
\usepackage{setspace} 

\usepackage[utf8]{inputenc}
\usepackage[T1]{fontenc}

\usepackage{amssymb,amsmath,amsthm,graphicx,xcolor,mathtools,mathrsfs}
\usepackage{comment}

\usepackage{hyperref}
\hypersetup{colorlinks, linkcolor={red!50!black}, citecolor={green!50!black}, urlcolor={blue!50!black}}

\usepackage[
backend=biber,
natbib=true,
style=numeric,
isbn=false,
maxbibnames=9
]{biblatex}
\usepackage{csquotes}
\renewbibmacro{in:}{}
\DeclareFieldFormat{pages}{#1}

\DeclareFieldFormat{title}{\mkbibquote{#1}}
\DeclareFieldFormat[misc]{date}{Preprint (#1)}

\newcommand{\calC}{\mathcal{C}}
\newcommand{\calH}{\mathcal{H}}
\newcommand{\calM}{\mathcal{M}}
\newcommand{\calP}{\mathcal{P}}
\newcommand{\calQ}{\mathcal{Q}}
\newcommand{\calR}{\mathcal{R}}
\newcommand{\calX}{\mathcal{X}}
\newcommand{\calZ}{\mathcal{Z}}

\renewbibmacro*{doi+eprint+url}{
	\iftoggle{bbx:url}
	{\iffieldundef{doi}{\usebibmacro{url+urldate}}{}}{}
	\newunit\newblock
	\iftoggle{bbx:eprint}
	{\usebibmacro{eprint}}
	{}
	\newunit\newblock
	\iftoggle{bbx:doi}
	{\printfield{doi}}
	{}
}

\usepackage{cleveref}
\theoremstyle{plain}
\newtheorem{theorem}{Theorem}[section]
\crefname{theorem}{Theorem}{Theorems}

\newtheorem{proposition}[theorem]{Proposition}
\crefname{proposition}{Proposition}{Propositions}

\newtheorem{corollary}[theorem]{Corollary}
\crefname{corollary}{Corollary}{Corollaries}

\newtheorem{lemma}[theorem]{Lemma}
\crefname{lemma}{Lemma}{Lemmas}

\crefname{conjecture}{Conjecture}{Conjectures}

\crefname{problem}{Problem}{Problem}

\newtheorem{claim}[theorem]{Claim}
\crefname{claim}{Claim}{Claims}

\crefname{observation}{Observation}{Observations}

\crefname{setup}{Setup}{Setups}

\crefname{myth}{Myth}{Myths}

\crefname{fact}{Fact}{Facts}

\crefname{algorithm}{Algorithm}{Algorithms}

\crefname{remark}{Remark}{Remarks}

\crefname{example}{Example}{Examples}

\theoremstyle{definition}
\newtheorem{definition}[theorem]{Definition}
\crefname{definition}{Definition}{Definitions}

\crefname{construction}{Construction}{Constructions}

\newtheorem{question}[theorem]{Question}
\crefname{question}{Question}{Questions}

\numberwithin{equation}{section}

\usepackage[shortlabels]{enumitem}
\setlist[enumerate,1]{label={\upshape (\roman*)}}

\DeclareMathOperator{\probability}{Pr}
\DeclareMathOperator{\expectation}{\mathbf{E}}

\DeclareMathOperator{\wsp}{wsp}
\DeclareMathOperator{\ssp}{ssp}
\newcommand{\eps}{\varepsilon}
\allowdisplaybreaks

\newenvironment{proofclaim}[1][Proof of the claim]{\begin{proof}[#1]}{\end{proof}}

\author[C. G. Fernandes]{Cristina G. Fernandes}

\author[G. O. Mota]{Guilherme Oliveira Mota}
\address[C. G. Fernandes and G. O. Mota]{Instituto de Matemática e
  Estatística, Universidade de São Paulo, Rua do Matão 1010,
  05508-090 São Paulo, Brazil}
\email{\{\,cris\,|\,mota\,\}@ime.usp.br}

\author[N.~Sanhueza-Matamala]{Nicol\'as Sanhueza-Matamala}
\address[N.~Sanhueza-Matamala]{Departamento de Ingeniería Matemática,
  Facultad de Ciencias Físicas y Matemáticas, Universidad de
  Concepción, Chile.}  \email{nsanhuezam@udec.cl}

\thanks{\thanks{This research was partly supported by a joint project
    FAPESP and ANID (2019/13364-7) and by CAPES (Finance Code 001). 
    C.~G.~Fernandes was supported by CNPq (310979/2020-0). 
    G.~O.~Mota was supported by CNPq (306620/2020-0, 406248/2021-4) and FAPESP
    (2018/04876-1). N. Sanhueza-Matamala was supported by ANID FONDECYT Iniciación Nº11220269 grant.  CAPES
    is the Coordena\c c\~ao de Aperfei\c coamento de Pessoal de
    N\'ivel Superior.  CNPq is the National Council for Scientific and
    Technological Development of Brazil.  FAPESP is the S\~ao Paulo
    Research Foundation. ANID is the Chilean National Agency for Research and Development.}}

\title{Separating path systems in complete graphs}

\begin{document}
\onehalfspace

\begin{abstract}
  We prove that in any $n$-vertex complete graph there is a collection
  $\mathcal{P}$ of $(1 + o(1))n$ paths that \emph{strongly separates}
  any pair of distinct edges $e, f$, meaning that there is a path in
  $\mathcal{P}$ which contains $e$ but not $f$.
  Furthermore, for certain classes of $n$-vertex $\alpha n$-regular graphs 
  we find a collection of $(\sqrt{3 \alpha + 1} - 1 + o(1))n$ paths
  that strongly separates any pair of edges.
  Both results are best-possible up to the $o(1)$~term.
\end{abstract}

\maketitle

\section{Introduction}

\subsection{Separating path systems}

Let $\mathcal{P}$ be a family of paths in a graph $G$.  We say that
two edges $e, f$ are \emph{weakly separated by $\mathcal{P}$} if there
is a path in $\mathcal{P}$ which contains one of these edges but not
both.  We also say that they are \emph{strongly separated by
	$\mathcal{P}$} if there are two paths $P_e, P_f \in \mathcal{P}$
such that $P_e$ contains $e$ but not $f$, and $P_f$ contains $f$ but
not $e$.

We are interested in the problem of finding ``small'' families of
paths (``path systems'') that separate any pair of edges in a given graph.
A path system in a graph~$G$ is \emph{weak-separating}
(resp.~\emph{strong-separating}) if all pairs of edges in $G$ are
weakly (resp.~strongly) separated by it. Let $\wsp(G)$ and $\ssp(G)$,
respectively, denote the size of smallest such families of paths in a
graph $G$.  Since every strong-separating path system is also
weak-separating, the inequality $\wsp(G) \leq \ssp(G)$ holds for any
graph~$G$, but equality is not true in general.

The study of general separating set systems was initiated by
Rényi~\cite{Renyi1961} in the 1960s.  The variation which considers the separation
of edges using subgraphs has been considered many times in the computer
science community, motivated by the application of efficiently
detecting faulty links in
networks~\cite{HPWYC2007,HonkalaKarpovskyLitsyn2003,TapolcaiRonyaiHo2013,ZakrevskiKarpovsky1998}.
The question got renewed interest in the combinatorics community after
it was raised by Katona in a conference in 2013, and was considered
simultaneously by Falgas-Ravry, Kittipassorn, Korándi, Letzter, and
Narayanan~\cite{FKKLN2014} in its weak version, and by Balogh, Csaba,
Martin, and Pluhár~\cite{BCMP2016} in its strong version.  Both teams
conjectured that $n$-vertex graphs $G$ admit (weak and strong)
separating path systems of size linear in $n$, that is, $\wsp(G),
\ssp(G) = O(n)$, and both also observed that an $O(n \log n)$ bound
holds. 
Letzter~\cite{Letzter2022} made substantial progress in this question by showing that all $n$-vertex graphs $G$ satisfy $\ssp(G) = O(n \log^\ast n)$.
The conjecture was settled by Bonamy,
Botler, Dross, Naia, and Skokan~\cite{BBDNS2023}, who proved that
$\ssp(G) \leq 19n$ holds for any $n$-vertex graph $G$.

\subsection{Separating cliques}

An interesting open question is to replace the value of `$19$' in
$\ssp(G)\leq 19n$ by the smallest possible number. Perhaps, it could
be possible even to replace this value by $1 + o(1)$.  Studying
separating path systems in complete graphs is particularly relevant, since
$K_n$ gives the best-known lower bounds for $\wsp(G)$ and
$\ssp(G)$ over all $n$-vertex graphs $G$ (see Section~\ref{section:conclusion} for further discussion). Because of this fact, the
behaviour of $\ssp(K_n)$ and $\wsp(K_n)$ has been enquired repeatedly by
many authors (e.g. \cite[Section 7]{FKKLN2014}).

For the weak separation, we know that $\wsp(K_n) \geq n-1$ (see the
remark before Conjecture 1.2 in~\cite{FKKLN2014}).  For strong
separation, mimicking that proof shows that the slightly better bound
$\ssp(K_n) \geq n$ holds (this is done
in~\Cref{proposition:lowerboundclique}).  Our first main result shows
that this lower bound is asymptotically correct.
\begin{theorem} \label{theorem:completegraph}
	The following holds.
	\[\ssp(K_n) = (1 + o(1))n.\]
\end{theorem}

Let us summarise the history of upper bounds for this problem now.
First we knew that $\wsp(K_n) = O(n)$ \cite[Theorem
1.3]{FKKLN2014}, and then $\ssp(K_n) \leq 2n+4$ \cite[Theorem
3]{BCMP2016}.  Wickes~\cite{Wickes2022} studied $\wsp(K_n)$ in more
detail, and showed that $\wsp(K_n) \leq n$ whenever $n$ or $n-1$ is a
prime number, and that $\wsp(K_n) \leq (21/16 + o(1))n$ in general.

The problem of estimating $\ssp(K_n)$ is connected with the older
problem of finding \emph{orthogonal double covers} (ODC), which are
collections $\mathcal{C}$ of subgraphs of~$K_n$ in which every edge
appears in exactly two elements of $\mathcal{C}$, and the intersection
of any two elements of $\mathcal{C}$ contains exactly one edge.
If each graph in $\mathcal{C}$ is isomorphic to some graph $H$, we say that $\mathcal{C}$ is an \emph{ODC by $H$}.
If $H$ is an $n$-vertex path and $\mathcal{C}$ is an ODC by $H$, then each element of~$\mathcal{C}$ is a Hamiltonian path, and it is easy
to check that $\mathcal{C}$ must contain exactly $n$ paths and forms a strong-separating path system. 
Thus, we know that $\ssp(K_n) = n$ whenever an ODC by Hamiltonian paths exists.  
This is known to be false for $n = 4$, but is known to be true for infinitely 
many values of $n$.  In particular, it holds if $n$ can be written as a
product of the numbers $5$, $9$, $13$,~$17$, and $29$~\cite{HortonNonay1991}.
See 
the survey~\cite{GGHLL2002} for more results and details.  
Gronau,
Müllin, and Rosa~\cite{GronauMullinRosa1997} conjectured that an ODC by $H$
in $K_n$ can be found whenever $H$ is any $n$-vertex tree which is not a path with three
edges. If true, this would imply that $\ssp(K_n) = n$ holds for every
$n \neq 4$.  An approximate version of this conjecture (obtained as a
corollary of general results about `rainbow trees') was obtained by
Montgomery, Pokrovskiy and Sudakov~\cite[Theorem~1.7]{MontgomeryPokrovskiySudakov2020} 
whenever $n$ is a large power of two.

\subsection{Separating regular graphs}
Our main result for cliques (\Cref{theorem:completegraph}) follows
from a more general result which works for ``robustly connected''
graphs which are \emph{almost regular}, meaning that each vertex has
approximately the same number of neighbours.  For simplicity, we give
the statement for regular graphs here.  Let $\alpha \in [0,1]$ and
consider an $\alpha n$-regular graph $G$ on $n$ vertices.  A counting
argument (\Cref{proposition:lowerboundgeneral}) shows that $\ssp(G)
\geq (\sqrt{3 \alpha + 1} - 1 - o(1))n$ must hold.  Our second main
result shows that this bound essentially holds with equality if we
also assume some vertex-connectivity condition. We say an $n$-vertex
graph $G$ is \emph{$(\delta, L)$-robustly-connected} if, for every $x,
y \in V(G)$, there exists $1 \leq \ell \leq L$ such that there are at
least $\delta n^\ell$ $(x,y)$-paths with exactly $\ell$ inner vertices
each.

\begin{theorem} \label{theorem:regular}
  Let $\alpha, \delta \in (0,1)$ and $L \geq 1$.
  Suppose that $G$ is an $n$-vertex graph which is $\alpha n$-regular
  and $(\delta, L)$-robustly-connected. Then
  \[\ssp(G) = (\sqrt{3\alpha + 1} - 1 + o(1))n.\]
\end{theorem}

We note that at least some kind of connectivity is required for a
bound like the one in Theorem~\ref{theorem:regular}. Indeed, the graph
$G$ formed by two vertex-disjoint cliques with $n/2$ vertices is $(n/2
- 1)$-regular, but has clearly $\ssp(G) = 2\cdot\ssp(K_{n/2}) \geq n$,
whereas \Cref{theorem:regular} would give an incorrect upper bound
around $(0.582 + o(1))n$.

Observe that the function $f(\alpha) = \sqrt{3 \alpha + 1} - 1$ satisfies
$\alpha < f(\alpha) < \sqrt{\alpha} < 1$ for {$\alpha \in (0,1)$}, so in
particular this shows that all $n$-vertex graphs $G$ covered by
\Cref{theorem:regular} satisfy $\ssp(G) \leq (1 + o(1))n$. From
\Cref{theorem:regular} we can obtain as corollaries results for many
interesting classes of graphs as balanced complete bipartite graphs, regular
graphs with large minimum degree, regular robust expanders, etc. 
(see Section~\ref{sec:corollaries} for details).

\subsection{Outline of the proof}

We summarise the idea behind our proof by focusing on the case of
estimating $\ssp(K_n)$.  The calculations which give the lower bound
$\ssp(K_n) \geq n$ (\Cref{proposition:lowerboundclique}) reveal that,
if $\ssp(K_n) = n$ holds, then in an optimal strong-separating path
system each path must be Hamiltonian, each edge must be covered
precisely by two paths, and every two different paths intersect
precisely in one edge.  Guided by this, our approach can be thought
conceptually of taking $t = (1 + \eps)n$ different `labels' and
finding an injective assignment $\phi: E(K_n) \rightarrow
\binom{t}{2}$ where every edge gets two labels.  Then, by defining,
for each $1 \leq i \leq t$, the subgraph $Q_i \subseteq K_n$
consisting of the edges which received label $i$, we get that the
family $\{Q_i\}_{1 \leq i \leq t}$ will strongly separate the edges of
$K_n$.

To make sure that the graphs $Q_i$ resemble paths, we will obtain the
assignment $\phi$ in a more careful way.  We will construct $\phi$
with the help of an almost perfect matching in an auxiliary hypergraph
$\mathcal{H}$.  In this case, the hypergraph can be described as
follows.  Orient the edges of $K_n$ to obtain a digraph $D$, then
obtain an auxiliary graph $B$ by taking two copies $V_1, V_2$ of
$V(K_n)$ and adding an edge between $u_1 \in V_1$ and $v_2 \in V_2$ if
the arc $(u,v)$ appears in $D$.
Next, consider a clique $K$ on a set of vertices $\{1, 2, \dotsc,
n\}$, vertex-disjoint from $V_1 \cup V_2$.  Form a graph $Z$ by adding
every edge between a vertex of $V_1 \cup V_2$ and a vertex in $K$.
Then, if $u_1 \in V_1, v_2 \in V_2, i, j \in V(K)$ and these vertices
form a $K_4$ in $Z$ (say those copies of $K_4$ are `valid'), we can
interpret that as assigning the edge $uv \in E(K_n)$ to the graphs
$Q_i$ and $Q_j$.  Similarly, if we have edge-disjoint valid copies of
$K_4$ in $Z$, this can be interpreted as assigning different edges of
$E(K_n)$ to different pairs $Q_i, Q_j$; without repeating pairs, and
assigning at most two edges adjacent to the same vertex in $K_n$ to
the same $Q_i$.  Thus if we can find edge-disjoint valid copies of
$K_4$ which use all edges between $V_1$ and $V_2$, we would have
obtained an allocation of all the edges of $E(K_n)$ to pairs of $Q_i,
Q_j$, where each $Q_i$ has maximum degree at most $2$.  To find such
edge-disjoint copies of $K_4$ we look at the auxiliary $6$-graph
$\mathcal{H}$ with vertex set $E(Z)$ and each valid copy of $K_4$
corresponding to an edge in $\mathcal{H}$.  By construction, an almost
perfect matching in $\mathcal{H}$ will yield graphs $\{Q_i\}_{1 \leq i
  \leq t}$ which separate `almost all' pairs of edges and have the
crucial property that $\Delta(Q_i) \leq 2$ for each~$i$.  This will
ensure that the graphs $Q_i$ are collections of paths and cycles.  To
find such an almost perfect matching in $\mathcal{H}$ we will use a
recent powerful result on hypergraph matchings by Glock, Joos, Kim,
Kühn and Lichev~\cite{GJKKL2022}, which will allow us to gain even
more control of the shape of the graphs $Q_i$ by avoiding certain
undesirable short cycles.

After this is done, we will have covered and separated most, but not
all, of the edges of $K_n$ with the graphs $\{Q_i\}_{1 \leq i \leq t}$, which are collections of
paths and cycles.  In a next step, we will transform each $Q_i$ by merging (most of) the edges of $Q_i$ into a single path
$Q'_i$.  This is done carefully to ensure the path system $\mathcal{Q}
= \{Q'_i\}_{1 \leq i \leq t}$ still strongly separates most of the
edges of the graph.

In the final step, the subgraph $H \subseteq G$ of edges which remain
unseparated will be very sparse, and in particular, has very small
maximum degree (at most $\eps n$).  Using a probabilistic
argument based on the Lovász Local Lemma, we find a small (of size $O(\eps
n)$) strong-separating path system $\mathcal{P}$ which
strongly separates $H$. Then our final desired path
system will be given by $\mathcal{P} \cup \mathcal{Q}$.

In this sketch of the proof we glossed over some details.  In the
actual proof (which covers the general case for $G \neq K_n$) the
situation is slightly more technical because in an optimal solution
the edges of $G$ need to be covered by $2$ or $3$ paths (as can be
seen from \Cref{proposition:lowerboundgeneral}).  The outline of the
proof is the same, but instead we will use a more intricate
auxiliary hypergraph $\mathcal{H}$ (in fact, we use an $8$-uniform
graph) to find the initial assignment.

\subsection{Organization of the paper}

In \Cref{section:lowerbounds} we give simple counting arguments which
yield the lower bounds in \Cref{theorem:completegraph} and
\Cref{theorem:regular}.
Then we begin the proof of our main result.
In \Cref{section:preliminaries} we gather some probabilistic and hypergraph tools and prove results that will be helpful during the next sections.
In \Cref{section:almostall} we find a family of graphs which separates almost all edges of a graph $G$, via a perfect matching in an auxiliary hypergraph.
In \Cref{section:breakingbad} we transform the given families of graphs into paths, keeping some structural properties.
In \Cref{sec:separating} we find small path systems which separate the remaining leftover edges.
Then the pieces of the main proofs are put together in \Cref{sec:mainresult}.

In \Cref{sec:corollaries} we describe how to use our main result in
some important graph classes, and we finalise with concluding remarks
in \Cref{section:conclusion}.

\section{Lower bounds} \label{section:lowerbounds}

Given a path system $\mathcal{P}$ in a graph $G$ and $e \in E(G)$, let
$\mathcal{P}(e) \subseteq \mathcal{P}$ be the paths of $\mathcal{P}$
which contain $e$.  Note that $\mathcal{P}$ is weak-separating if and
only if the sets $\mathcal{P}(e)$ are different for all $e \in E(G)$;
and $\mathcal{P}$ is strong-separating if and only if no
set $\mathcal{P}(e)$ is contained in another $\mathcal{P}(f)$.

\begin{proposition} \label{proposition:lowerboundclique}
  For each $n \geq 3$, $\ssp(K_n) \geq n$.
\end{proposition}

\begin{proof}
  Let $n \geq 3$.  Let $\mathcal{P}$ be a strong-separating path
  system in $K_n$, and define $\mathcal{P}_1 = \{ P \in \mathcal{P} :
  |E(P)| = 1 \}$.  Note that
  \[\sum_{e \in E(K_n)} |\mathcal{P}(e)| = \sum_{P \in \mathcal{P}} |E(P)| \leq |\mathcal{P}_1| + |\mathcal{P} \setminus \mathcal{P}_1|(n-1) = |\mathcal{P}| (n-1) - |\mathcal{P}_1|(n-2),\]
  where we used that each path can contain at most $n-1$ edges.
  
  Next, let $E_1 \subseteq E(K_n)$ be the set of edges $e$ such that
  $|\mathcal{P}(e)| = 1$.  Note that $|E_1| \leq |\mathcal{P}_1|$,
  because if an edge $e$ is covered by an unique path $P$, then $P$
  cannot cover any other edge~$f$, as otherwise there would be no
  other path which covers $e$ and not $f$, a contradiction to the fact
  that $\mathcal{P}$ is strong-separating.  We have that
  \[ \sum_{e \in E(K_n)} |\mathcal{P}(e)| \geq |E_1| + 2 \left( \binom{n}{2} - |E_1| \right) = n(n-1) - |E_1| \geq n(n-1) - |\mathcal{P}_1|, \]
  which implies that
  \[ n(n-1) \leq |\mathcal{P}|(n-1) - |\mathcal{P}_1|(n-3) \leq |\mathcal{P}|(n-1), \]
  where the last inequality uses $n \geq 3$.
  This implies that $|\mathcal{P}| \geq n$.
\end{proof}

\begin{proposition} \label{proposition:lowerboundgeneral}
  For any $\alpha, \eps \in (0,1]$, the following holds for all sufficiently large $n$.
  Let $G$ be an $n$-vertex graph with $\alpha \binom{n}{2}$ edges.
  Then
  \[\ssp(G) \geq (\sqrt{3 \alpha + 1} - 1 - \eps)n.\]
\end{proposition}

\begin{proof}
	Let $\alpha, \eps$ be given, and suppose $n$ is sufficiently large.
	Given $G$ as in the statement, let~$\mathcal{P}$ be a
        strong-separating path system of size $\ssp(G)$.
	Suppose $\beta$ is such 
  that $|\mathcal{P}| = \beta n$ (we know that $\beta \leq 19$ by the result of \cite{BBDNS2023}).
  We need to show that $\beta \geq \sqrt{3 \alpha + 1} - 1 - \eps$.
  Note that
  \begin{equation*}
  	\sum_{e \in E(G)} | \mathcal{P}(e)| = \sum_{P \in \mathcal{P}}
  	|E(P)| \leq \beta n(n-1) = 2 \beta \binom{n}{2}.
  \end{equation*}
  For $i \in \{1, 2\}$, let $E_i \subseteq E(G)$ be the set of edges
  $e$ such that $|\mathcal{P}(e)| = i$.  Then
  \begin{align}
  	2 \beta \binom{n}{2} & \geq \sum_{e \in E(G)} |\mathcal{P}(e)| \geq |E_1| + 2 |E_2| + 3
  	\left( \alpha \binom{n}{2} - |E_1| - |E_2| \right) \nonumber \\
  	& = 3 \alpha \binom{n}{2} - 2 |E_1| - |E_2|. \label{eq:lower2}
  \end{align}
  Since $\calP$ is strong-separating, if $e\in E_2$, then
  the two paths of $\calP$ that contain $e$ cannot both contain any
  other edge $f\in E_2$. Thus, $|E_2| \leq \binom{|\mathcal{P}|}{2} \leq
  \binom{\beta n}{2} \leq \beta^2 \binom{n}{2} + \beta^2n$. Note that we
  also have $|E_1| \leq |\mathcal{P}| \leq \beta n$. 
  Applying these bounds on $|E_1|$ and $|E_2|$ in~\eqref{eq:lower2}, we get
  \[
  \beta^2 \binom{n}{2} +  2\beta \binom{n}{2} 
  \geq 3 \alpha \binom{n}{2} - \beta^2n - 2\beta n
  \geq 3 \alpha \binom{n}{2}- 400n,
  \]
  where in the last step we used $\beta \leq 19$ to get $\beta^2 n + 2
  \beta n \leq 400n$.  Thus the inequality $\beta^2 + 2\beta \geq 3
  \alpha - 800/n$ holds. Since $\beta > 0$ and $n$ is sufficiently
  large, solving this quadratic equation in terms of $\beta$ gives
  that $\beta \geq \sqrt{3 \alpha + 1} - 1 - \eps$, as desired.
\end{proof}

\section{Preliminaries} \label{section:preliminaries}

\subsection{Hypergraph matchings}

We use a recent result by Glock, Joos, Kim, Kühn, and
Lichev~\cite{GJKKL2022} (similar results were obtained also by
Delcourt and Postle~\cite{DelcourtPostle2022}).  This result allows us
to find almost perfect matchings in hypergraphs $\calH$ which avoid
certain `conflicts'.  Each conflict is a subset of edges $X \subseteq
E(\calH)$ which we do not want to appear together in the matching $M$,
i.e., we want $X \not\subseteq M$ for all such conflicts $X$.  We
encode these conflicts using an auxiliary `conflict hypergraph'
$\calC$ whose vertex set is $E(\calH)$ and each edge is a different
conflict, i.e., each edge of $\calC$ encodes a set of edges of
$\calH$.

Given a (not necessarily uniform) hypergraph $\calC$ and $k \geq 1$, 
let $\calC^{(k)}$ denote the subgraph of $\calC$ consisting of all edges 
of size exactly $k$.  
If $\calC = \calC^{(k)}$, then $\calC$ is a \emph{$k$-graph}. 
For a hypergraph $\mathcal{H}$ and $j \geq 1$, 
let $\delta_j(\mathcal{H})$ (resp.\ $\Delta_j(\mathcal{H})$) be the 
minimum (resp.\ maximum) of the number of edges of $\mathcal{H}$ which contain $S$,
taken over all $j$-sets $S$ of vertices.
We say that a hypergraph $\mathcal{H}$ is \emph{$(x \pm y)$-regular} if $x - y \leq \delta_1(\mathcal{H}) \leq \Delta_1(\mathcal{H}) \leq x
+ y$.  
Let $N_{\calH}(v)$ denote the set of neighbours of $v$ in~$\calH$. 
Given a hypergraph $\mathcal{C}$ with $V(\mathcal{C}) = E(\mathcal{H})$, 
we say $E \subseteq E(\mathcal{H})$ is \emph{$\mathcal{C}$-free} 
if for every $C \in E(\mathcal{C})$, $C$ is not a subset of $E$. 
Also, $\mathcal{C}$ is a \emph{$(d, \ell, \rho)$-bounded conflict system} for $\calH$ if
\begin{enumerate}[(C1),labelindent=0pt,labelwidth=\widthof{\ref{c1last}},leftmargin=!]
	\item $3 \leq |C| \leq \ell$ for each $C \in \mathcal{C}$;
	\item $\Delta_1(\mathcal{C}^{(j)}) \leq \ell d^{j-1}$ for all $3 \leq j
	\leq \ell$; and 
	\item $\Delta_{j'}(\mathcal{C}^{(j)}) \leq \ell d^{j-j'-\rho}$ for all
	$3 \leq j \leq \ell$ and $2 \leq j' < j$.\label{c1last}
\end{enumerate}
We say that a set of edges $Z \subseteq E(\mathcal{H})$ is 
\emph{$(d, \rho)$-trackable}\footnote{ This corresponds to the $j=1$ and $\varepsilon = \rho$ case of the definition of \emph{$(d, \varepsilon, \mathcal{C})$-trackable test systems} of~\cite[Section 3]{GJKKL2022}. The original definition requires more properties but reduces to the definition we have given when $j=1$. In particular, $\mathcal{C}$ does not play a role anymore, so we opted for removing it from the notation.} if $|Z| \geq d^{1+\rho}$.

\begin{theorem}[{\cite[Theorem 3.2]{GJKKL2022}}] \label{theorem:hypergraphmatching}
	For all $k, \ell \geq 2$, there exists $\rho_0 > 0$ such that for all 
	$\rho \in (0, \rho_0)$ there exists $d_0$ so that the following holds for all $d \geq d_0$.
	Suppose $\calH$ is a $k$-graph on $n \leq \exp(d^{\rho^3})$ vertices with 
	$(1 - d^{-\rho})d \leq \delta_1(\calH) \leq \Delta_1(\calH) \leq d$ 
	and $\Delta_2(\calH) \leq d^{1 - \rho}$ and suppose $\calC$ is 
	a $(d, \ell, \rho)$-bounded conflict system for $\calH$.
	Suppose $\mathcal{Z}$ is a set of $(d, \rho)$-trackable 
	sets of edges in $\calH$ with ${|\mathcal{Z}| \leq \exp(d^{\rho^3})}$.
	Then, there exists a $\calC$-free matching $\calM \subseteq \calH$ of size at least 
	${(1 - d^{-\rho^3})n/k}$ with $|Z \cap \calM| = (1 \pm d^{-\rho^3}) |\calM| |Z| / |E(\calH)|$ 
	for all $Z \in \mathcal{Z}$.
\end{theorem}

\subsection{Counting cycles}

Let $D_n$ be the complete digraph (having all arcs in both directions).
The following lemma is a simple counting argument that will be used later.

\begin{lemma} \label{lemma:countingcycles}
	If $R \subseteq E(D_n)$ has $\ell < j$ edges, then there are at most~$j^\ell n^{j - \ell - 1}$ 
	length-$j$ directed cycles in $D_n$ which contain $R$.
\end{lemma}

\begin{proof}
We can assume that $R$ is a proper subgraph of some directed cycle (as
otherwise there is nothing to count).  Thus, $R$ is a collection of
vertex-disjoint paths in $D_n$ with exactly~$\ell$ edges in total.  All
directed cycles on $j$ vertices which contain $R$ can be obtained by
assigning a number in $\{1, \dotsc, j\}$ to allocate the starting
position for each of the paths in the cycle, and then choosing each of
the $j - |V(R)|$ remaining vertices.  Note that $R$ can consist of at
most $\ell$ paths, so the first step can be done in at most $j^\ell$
ways.  On the other hand, $R$ spans at least $\ell + 1$ vertices (the
minimum number occurs when $R$ is a single path), so there are at most $n^{j
  - \ell - 1}$ ways to choose the vertices in the second step.
Therefore, $R$ is contained in at most~$j^\ell n^{j - \ell - 1}$
length-$j$ directed cycles in $D_n$.
\end{proof}

\subsection{Probabilistic tools}

In this short section we state some standard probabilistic tools used
in our proof. 

\begin{lemma}[Chernoff's inequalities~{\cite[Remark 2.5, Corollaries 2.3 and 2.4]{JLR00}}]
  \label{lemma:chernoff}
  Let $X$ be a random variable with binomial distribution $B(n,p)$.
  Let $t \geq 0$.
  Then, 
  \begin{enumerate}
  	\item \label{chernoff:standard} $\probability[|X - \expectation[X]| \geq t] \leq
  	2\exp(-2t^2/n)$;
  	\item \label{chernoff:smallexpectation} if $t \leq 3 \expectation[X] / 2$, then
  	$\probability[|X - \expectation[X]| \geq t] \leq
  	2\exp(-{t^2}/({3\expectation[X]}))$; and
  	\item \label{chernoff:verysmallexpectation} if $t \geq 7
          \expectation[X]$, then $\probability[X \geq t] \leq
          \exp(-t)$.
  \end{enumerate}
\end{lemma}

The following concentration inequality will also be useful.

\begin{lemma}[McDiarmid's inequality~\cite{McDiarmid1989}]
    \label{lemma:mcdiarmid}
  Let $X_1, \ldots, X_M$ be independent random variables, with $X_i$
  taking values on a finite set $A_i$ for each $i \in [M]$. Suppose
  that $f:\prod_{j=1}^{M} A_j \rightarrow \mathbb{R}$ satisfies $|f(x)
  - f(x')| \leq c_i$ whenever the vectors $x$ and $x'$ differ only in
  the $i$th coordinate, for every $i \in [M]$. Consider the random variable $Y =
  f(X_1, \ldots, X_M)$ and $t \geq 0$.
  Then
  \begin{equation*}
    \probability[|Y - \expectation(Y)| \geq t] \leq 2 \exp\left(-
    \frac{2t^2}{\sum_{j=1}^M {c_j}^2}\right).
  \end{equation*}
\end{lemma}



\subsection{Building a base hypergraph}

The next lemma constructs an auxiliary hypergraph which we will 
use as a base to apply \Cref{theorem:hypergraphmatching} later.

\begin{lemma} \label{lemma:basehypergraph}
	For any $\alpha, \beta, \lambda > 0$ with $\beta = \sqrt{3 \alpha + 1} - 1 < \lambda$, 
	there exists $n_0$ such that the following holds for every $n \geq n_0$.  
	There exists a $3$-graph $J$ such that
	\begin{enumerate}[label=\upshape{\textbf{(J\arabic*)}},labelindent=0pt,labelwidth=\widthof{\ref{item:J-regular}},leftmargin=!]
		\item \label{item:J-vpartition} there is a partition $\{U_1, U_2\}$
		of $V(J)$ with $|U_1| = \lambda n$ and ${|U_2| = \lambda n \beta / 2}$;
		\item \label{item:J-2or3}there is a partition $\{J_1, J_2\}$ of $E(J)$ such that
		\begin{itemize}
			\item $e \subseteq U_1$ for each $e \in J_1$, and
			\item $|e \cap U_1| = 2$ for each $e \in J_2$;
		\end{itemize}
		\item \label{item:J-antichain} every pair $\{i,j\} \subseteq U_1$ is
		contained only in edges of $J_1$, or in at most one edge of~$J_2$;
		\item \label{item:J-codegree} $\Delta_2(J) \leq \ln^2 n$;
		\item \label{item:J-totalnumber} $J$ has $\alpha \binom{n}{2} \pm
		n^{2/3}$ edges in total; and
		\item \label{item:J-regular} $J$ is $(\beta n / \lambda, n^{2/3})$-regular.
	\end{enumerate}
\end{lemma}

\begin{proof}
	We begin by observing that $3 \alpha = 2 \beta + \beta^2$ 
	from the definition of $\beta$. Then, defining $d_2 = \beta^2n/\lambda$ 
	and $d_3 = 3(\alpha-\beta^2)n/(2\lambda)$, we obtain $d_2 + d_3 = \beta n/\lambda$.
	A \emph{$\{2,3\}$-graph} is a hypergraph whose edges have size either $2$ or~$3$. 
	We say it is an \emph{antichain} if no edge is contained in another.  
	We begin our construction by defining an antichain $\{2,3\}$-graph 
	on a set of size $\lambda n$.  Let $U_1$ be a set of $\lambda n$ vertices.  
	We claim that there is an antichain $\{2,3\}$-graph~$I$ on $U_1$ such that
	\begin{enumerate}[(F1),labelindent=0pt,labelwidth=\widthof{\ref{item:F-3}},leftmargin=!]
		\item \label{item:F-2} each vertex is adjacent to $d_2 \pm n^{2/3}$
		edges of size $2$ in $I$;
		\item \label{item:F-3} each vertex is adjacent to $d_3 \pm n^{2/3}$
		edges of size $3$ in $I$.
	\end{enumerate}
	
	Indeed, define a random graph $I^{(2)}$ on $U_1$ by including each
	edge independently with probability $p := \beta^2/\lambda^2 < 1$.
	Let $\overline{I}$ be the complement of $I^{(2)}$.  In
	expectation, each vertex is contained in around $p |U_1| = d_2$ many
	edges.
	A standard application of Chernoff's inequality (\Cref{lemma:chernoff}\ref{chernoff:standard}) shows that, with overwhelmingly large probability, each vertex of
	$U_1$ is contained in $d_2 \pm n^{2/3}$ edges of $I^{(2)}$, and thus
	we can assume that a choice of $I^{(2)}$ is fixed and satisfies that
	property. Similarly, we can also assume that every vertex is
	contained in $(1-p)^3 \binom{\lambda n}{2} \pm n^{4/3}$ triangles in
	$\overline{I}$.  Next, we form a $3$-graph $I^{(3)}$ on $U_1$ by including
	each triple of vertices which forms a triangle in $\overline{I}$
	with probability $q := 3 (\alpha - \beta^2) / ((1-p)^3 \lambda^3n)$.  If a vertex $x$ is contained in $d:=(1-p)^3 \binom{\lambda n}{2}$ 
	triangles in~$\overline{I}$, then in expectation it must be
	contained in $dq = d_3 \pm n^{1/2}$ many $3$-edges in~$I^{(3)}$.
	Using Chernoff again (\Cref{lemma:chernoff}\ref{chernoff:smallexpectation}) we can assume that each vertex in $I^{(3)}$ is contained in 
	$d_3 \pm n^{2/3}$ many triangles in $I^{(3)}$.  
	We conclude by taking $I = I^{(2)} \cup I^{(3)}$.
	
	Now we transform $I$ into a $3$-graph.  To achieve this, we will add
	a set $U_2$ of extra vertices to $I$, and extend each $2$-edge of
	$I$ to a $3$-edge by including to it a vertex in $U_2$.  Let $U_2$
	have size $r := \lambda n \beta / 2$ and vertices $\{v_1, \dotsc,
	v_r \}$.  Randomly partition the $2$-edges of $I^{(2)}$ into $r$
	sets $F_1, \dotsc, F_r$ by including each edge of~$I^{(2)}$ in 
	an $F_i$ with probability $1/r$.  Next, define sets of $3$-edges
	$H_1, \dotsc, H_r$ given by~$H_i := \{ xyv_i : xy \in F_i \}$.
	
	Let $J$ be the $3$-graph on vertex set $U_1 \cup U_2$ whose edges
	are $I^{(3)} \cup \bigcup_{i=1}^r H_i$.  Note that, by construction,
	$J$ satisfies \ref{item:J-vpartition}--\ref{item:J-antichain}, 
	so it only remains to verify \ref{item:J-codegree},
	\ref{item:J-totalnumber} and \ref{item:J-regular}.
	
	We show that \ref{item:J-codegree} holds.  Let $x, y$ be a pair of
	vertices in $V(J)$ and let us consider the possible cases.  If $x, y
	\in U_1$ and $xy \in I^{(2)}$, then $\deg_{J}(xy) = 1$, because its only
	neighbour is $v_i$ (if $xy \in F_i$).  If $x, y \in U_1$ and $xy \in
	\overline{I}$, then $\deg_{J}(xy)$ is precisely the number of
	triangles of $I^{(3)}$ that contain $xy$. This is a random variable
	with expected value at most $nq = O(1)$.  Thus, by \Cref{lemma:chernoff}\ref{chernoff:verysmallexpectation}, $\deg_{J}(xy) > \ln^2 n$ holds with
	probability at most $n^{- \ln n}$, so we can comfortably use the
	union bound to ensure that $\deg_{J}(xy) \leq \ln^2 n$ for every
	such pair $xy \in \overline{I}$.  If $x \in U_1$ and $y \in U_2$,
	then $y = v_i$ for some $1 \leq i \leq r$, and $\deg_{J}(xy)$ is the
	number of triangles of the form $xzv_i \in H_i$.  For a fixed $x$,
	there are at most $|U_1| = \lambda n$ choices for $z$ to form an
	edge $xz \in I^{(2)}$, and recall that each such edge belongs to
	$H_i$ with probability $1/r = O(1/n)$.  Thus the expected value of
	$\deg_{J}(xy)$ is again of the form $O(1)$, and we can conclude the
	argument in a similar way as before. Finally, if $x, y \in U_2$ 
	then $\deg_J(xy) = 0$ by construcion.  This finishes the proof 
	of~\ref{item:J-codegree}.
	
	Note that $|E(J)| = |E(I^{(2)})|+|E(I^{(3)})|$.  From \ref{item:F-2}, we deduce  
	that $|E(I^{(2)})|$ is {$\lambda n (d_2 \pm n^{2/3})/2 = \beta^2 n^2 /2
	\pm n^{2/3}$} and, from \ref{item:F-3}, we deduce that $|E(I^{(3)})|$ is
	{$\lambda n (d_3 \pm n^{2/3})/3 = (\alpha - \beta^2) n^2 /2 \pm
	n^{2/3}$}, so $|E(J)| = \alpha n^2 /2 \pm O(n^{2/3})$, which proves
	\ref{item:J-totalnumber}.
	
	Now we prove \ref{item:J-regular}.
	Let $i \in V(J)$. If $i \in U_1$, then $\deg_J(i) = \deg_I(i)$.
	Since $d_2 + d_3 = \beta n/\lambda$, we have that $\deg_I(i) = d_2 + d_3 +
	O(n^{2/3}) = \beta n / \lambda + O(n^{2/3})$.
	Assume now that $i \in U_2$.
	Recall that we defined $J$ in a way that
	each vertex of $U_2$ belongs to $|E(I^{(2)})|/r$ edges, and we have
	\begin{align*}
		\frac{|E(I^{(2)})|}{r}
		& = \frac{\beta^2 \binom{n}{2} \pm O(n^{4/3})}{\lambda n \beta / 2}
		= \frac{\beta n}{\lambda} \pm O(n^{1/3}),
	\end{align*} 
	which concludes the proof of the lemma.
\end{proof}

\section{Separating almost all edges} \label{section:almostall}

In this section we show how to separate most pairs of edges of
robustly-connected graphs by paths and cycles, guaranteeing additional
structural properties.

In what follows, let $\eps$, $\delta > 0$, let $L$ be an integer and
let $G$ be an $n$-vertex graph.
A \emph{2-matching} in $G$ is a collection of vertex-disjoint
cycles and paths in $G$. We say a $2$-matching $Q$ in $G$ is
\emph{$(\delta, L)$-robustly-connected} if, for every $x, y \in V(Q)$,
there exists $\ell$ with $1 \leq \ell \leq L$ such that there are at least $\delta
n^\ell$ $(x,y)$-paths with exactly $\ell$ inner vertices each, all in
$V(G) \setminus V(Q)$.  Furthermore, a collection $\calQ$ of
2-matchings in~$G$ is \emph{$(\delta, L)$-robustly-connected} if each
$Q\in\calQ$ is $(\delta, L)$-robustly-connected.  A 2-matching $Q$ in
$G$ is \emph{$\eps$-compact} if each cycle in $Q$ has length at
least~$1/\eps$ and $Q$ contains at most $\eps n$ paths.  For a
collection $\calQ$ of 2-matchings in~$G$, we say $\calQ$ is
\emph{$\eps$-compact} if each $Q$ in $\calQ$ is $\eps$-compact.

Let $\calQ = \{Q_1, \dotsc, Q_t\}$ be a collection of subgraphs of~$G$.  
We use $E(\calQ)$ to denote the set $\bigcup_{i=1}^t E(Q_i)$.  We say 
$\calQ$ \emph{separates an edge $e$ from all other edges of~$G$} if the 
set~$\{ i : e \in E(Q_i) \}$ is not contained in the set $\{ j : f \in E(Q_j) \}$ 
for each $f \in E(G) \setminus \{e\}$.  Clearly, if an edge $e$ is 
separated from all other edges of~$G$ by~$\calQ$, then $e \in E(\calQ)$.  
We also say that $\calQ$ \emph{strongly separates} a set~$E'$ of edges if, 
for every distinct $e,f \in E'$, the sets~$\{ i : e \in E(Q_i) \}$ 
and~$\{ j : f \in E(Q_j) \}$ are not contained in each other.

For brevity, we put together some of the above definitions in one
single concept, that will be used in the next result and also in some
lemmas in Section~\ref{sec:connecting} (see Lemmas~\ref{lem:acyclic}
and~\ref{lemma:connectingpaths}).

\begin{definition} \label{definition:separator}
  Given a graph $G$ and $\delta$, $L$, $\beta$ and $\eps$, we say a
  collection of $2$-matchings $\calQ$ is a \emph{$(\delta,
    L,\beta,\eps)$-separator} for $G$ if the following holds.
  \begin{enumerate}[\upshape{\textbf{(Q\arabic*)}},labelindent=0pt,labelwidth=\widthof{\ref{item:firstpacking-separating}},leftmargin=!]
  \item \label{item:firstpacking-connecting} $\calQ$ is $\eps$-compact and $(\delta,L)$-robustly-connected;
  \item \label{item:firstpacking-separating} $|\calQ| = \beta n$ and $\calQ$ strongly separates $E(\calQ)$;
  \item \label{item:firstpacking-endpoints} each vertex in $G$ is the endpoint of at most $\eps n$ paths among all $Q \in \calQ$;
  \item \label{item:firstpacking-sizes} each $e \in E(\calQ)$ is contained in at most three of the 2-matchings in $\calQ$; and
  \item \label{item:firstpacking-almostperfect}$\Delta(G - E(\calQ)) \leq \eps n$.
  \end{enumerate}
\end{definition}  

In the next result we show that large enough $(\delta, L)$-robustly-connected
``almost regular'' graphs contain a suitable collection of
$2$-matchings that is a $(\eps^\ell\delta/2, L,\beta/(1-\eps),\eps')$-separator, for any
$\eps$ and $\eps'$.

\begin{lemma} \label{lemma:firstpacking} Let $1/n \ll \eps, \eps',
  \alpha, \delta, 1/L, \rho$.  Let $\beta = \sqrt{3 \alpha + 1} - 1$.
  If $G$ is an $n$-vertex ${(\alpha n \pm n^{1 - \rho})}$-regular graph that is $(\delta,
  L)$-robustly-connected, then there exists a $(\eps^\ell\delta/2,
  L,\beta/(1-\eps),\eps')$-separator for $G$.
\end{lemma}

\begin{proof}
Our proof has five steps.  First, we build an auxiliary
hypergraph~$\mathcal{H}$ such that a large matching $M \subseteq
\mathcal{H}$, which avoids certain conflicts, yields a
family of subgraphs of $G$ with the desired properties.  We wish to apply
\Cref{theorem:hypergraphmatching} to find such a matching.  In the
second step, we verify that $\mathcal{H}$ satisfies the
hypotheses of \Cref{theorem:hypergraphmatching}. In the third step,
we define our conflict hypergraph $\calC$. In the fourth step, we
define some test sets and prove they are trackable. Having done
this, we are ready to apply \Cref{theorem:hypergraphmatching}, which
is done in the last step. Then we verify that the construction gives
the desired graphs.
From now, we can assume that $\rho$ is sufficiently small (since that only weakens our assumptions).
Also, we assume that $n$ is sufficiently large with respect to $\eps, \eps', \alpha, \delta, L, \rho$ so that every calculation that requires it is valid. \medskip

\noindent \emph{Step 1: Constructing the auxiliary hypergraph.}
Obtain an oriented graph $D$ by orienting each edge of $G$ uniformly
at random. Each vertex $v$ has expected in-degree and out-degree
$d_G(v)/2 = (\alpha n \pm n^{1 - \rho})/2$. So, by
Chernoff's inequality~(\Cref{lemma:chernoff}\ref{chernoff:standard}) and a union bound, we can assume that in $D$
every vertex has in-degree and out-degree of the form $\alpha n/2 \pm
2n^{1 - \rho}$.

Next, consider an auxiliary bipartite graph $B$ whose clusters are
copies $V_1$ and $V_2$ of $V(G)$, where each vertex $x \in V(G)$ is
represented by two copies $x_1 \in V_1$ and $x_2 \in V_2$,
and such that $x_1 y_2 \in E(B)$ if and only if $(x, y) \in E(D)$.
Thus we have that $|E(B)| = |E(G)| = \alpha \binom{n}{2} \pm n^{2 - \rho}$, 
because $G$ is $(\alpha n \pm n^{1 - \rho})$-regular. 
Finally, let $\lambda = \beta/(1 - \eps)$. 
Apply \Cref{lemma:basehypergraph} with $\alpha,\beta,\gamma$ to obtain 
a $3$-graph $J$ which satisfies \ref{item:J-vpartition}--\ref{item:J-regular} 
and assume that $U_1 = [\lambda n]$ and $V(J) = [|V(J)|]$.

Now we build an initial auxiliary $8$-graph $\mathcal{H}'$ as follows.  
Let $Z$ be the complete bipartite graph between clusters $V(B)$ and~$V(J)$.  
The vertex set of $\mathcal{H}'$ is ${E(B) \cup E(J) \cup E(Z)}$.  
Each edge in $\mathcal{H}'$ is determined by a
choice $x_1 y_2 \in E(B)$ and~$ijk \in E(J)$, which form an edge
together with the $6$ edges in $Z$ that join $x_1$ and $y_2$ to
$i$, $j$, and $k$. More precisely, the edge determined by $x_1 y_2
\in E(B)$ and $ijk \in E(J)$ is $\Phi(x_1 y_2, ijk) := \{x_1 y_2,
ijk, x_1 i, x_1 j, x_1 k, y_2 i, y_2 j, y_2 k \}$; and the edge set
of $\mathcal{H}'$ is given by
\[
E(\mathcal{H}') = \{ \Phi(x_1 y_2, ijk) : x_1 y_2 \in E(B), ijk
\in E(J) \}.
\]

The idea behind the construction of $\mathcal{H}'$ is as follows.
Suppose that $M$ is a matching in $\mathcal{H}'$, and that $x_1 y_2 \in
E(B)$ is covered by $M$ and appears together with $ijk \in E(J)$ in
an edge of~$M$.
By \ref{item:J-2or3}, $\{i,j,k\} \cap U_1$ has size $2$ or $3$.
Recall that we want to obtain a collection 
$\calQ := \{Q_1, \dotsc, Q_t\}$ of $2$-matchings in $G$, where $t = \lambda n$,
satisfying~\ref{item:firstpacking-connecting}--\ref{item:firstpacking-almostperfect}.
We will add edges $xy \in E(G)$ such that $\Phi(x_1y_2,ijk) \in M$ or 
$\Phi(y_1x_2,ijk) \in M$ to the graphs $Q_a$ if $a \in \{i,j,k\} \cap U_1$.  

By construction, and since $M$ is a matching, at most one edge in $B$ involving 
$x_1$ (resp.~$x_2$) appears in an edge of $M$ together with some $a \in U_1$.  
By considering the contributions of the two copies $x_1, x_2 \in V(B)$ 
of a vertex $x \in V(G)$, this means that the subgraphs $Q_a \subseteq G$ 
have maximum degree $2$, and thus these graphs are $2$-matchings in $G$, 
as we wanted. By construction and property~\ref{item:J-2or3}, 
each edge in $E(G)$ belongs to either $0$, $2$ or $3$ graphs $Q_a$.  
Importantly, property \ref{item:J-antichain} implies that, 
for two distinct edges $e, f \in E(G)$, no two non-empty sets of the type 
$\{a : e \in E(Q_a)\}$ and $\{b : f \in E(Q_b)\}$ can be contained in each other.
Straightforward calculations reveal the degrees of the vertices in $V(\mathcal{H}')$.

\begin{claim}
	The following hold.
	\begin{enumerate}
		\item $\deg_{\mathcal{H}'}(x_1 y_2) = \alpha \binom{n}{2} \pm n^{2/3}$ for each $x_1 y_2 \in E(B)$;
		\item $\deg_{\mathcal{H}'}(ijk) = \alpha \binom{n}{2} \pm n^{2 - \rho}$ for each $ijk \in E(J)$; and
		\item $\deg_{\mathcal{H}'}(x_a i) = \frac{ \alpha\beta}{\lambda} \binom{n}{2} \pm 2n^{2 - \rho}$ for each $x_a i \in E(Z)$.
	\end{enumerate}
\end{claim}

\begin{proofclaim}
	The first two points can be easily verified: given any $x_1 y_2 \in E(B)$, 
	by construction $d_{\mathcal{H}'}(x_1 y_2)$ is the number of edges in $J$, 
	which is $\alpha \binom{n}{2} \pm n^{2/3}$ by~\ref{item:J-totalnumber}.  
	Moreover, given any $ijk \in E(J)$, $d_{\mathcal{H}'}(ijk)$ is the number of 
	edges of $E(B)$, which is $\alpha \binom{n}{2} \pm n^{2 - \rho}$ by construction. 
	
	Finally, consider $x_a \in V_1$ (the case $x_a \in V_2$ is symmetric) and $i \in V(J)$.  
	The degree $\deg_{\mathcal{H}'}(x_a i)$ corresponds to the edges $\Phi(x_a y_2, ijk)$ 
	with $y_2 \in N_B(x_1)$ and $jk \in N_J(i)$.
	Next, we estimate the number of valid choices for $y_2$ and $jk$.
	There are $\deg_B(x_1) = d^+_D(x) = \alpha n / 2 \pm 2n^{1 - \rho}$ possible choices 
	for~$y_2$, and there are $\deg_J(i) = \beta n / \lambda \pm n^{2/3}$ possible choices 
	for $jk$ by~\ref{item:J-regular}.  Thus we deduce that 
	\[ \deg_{\mathcal{H}'}(x_1 i) = 
	\left( \frac{\alpha n}{2} \pm 2n^{1 - \rho} \right) \left( \frac{\beta n}{\lambda} \pm n^{2/3} \right) 
	= \frac{\alpha \beta}{\lambda} \binom{n}{2} \pm 2n^{2 - \rho}, \]
	as desired.
\end{proofclaim}

Since $\mathcal{H}'$ is not quite regular, we will actually work
with a carefully-chosen subgraph $\mathcal{H}$ of ${H}'$.  
Let
\(p
:= \beta/\lambda = 1 - \eps.
\)
For each $i \in V(J)$, select a subset $X_i \subseteq V(G)$ by including 
in $X_i$ each vertex of $G$ independently at random with probability $p$.  
This defines a family $\{X_i : i \in V(J)\}$ of subsets of $V(G)$. 
For each $x\in V(G)$, consider the random set $Y_x = \{i \in V(J) : x\in X_i\}$.
Finally, let $\mathcal{H} \subseteq \mathcal{H}'$ be the induced subgraph 
of $\mathcal{H}'$ obtained after removing all vertices $x_1 i, x_2 i \in E(Z)$ 
whenever $x \notin X_i$ (or, equivalently, $i\notin Y_x$). 
Thus, we have that
\[
E(\mathcal{H}) = \{ \Phi(x_1 y_2, ijk ) : x_1 y_2 \in E(B), ijk
\in E(J), \{x,y\} \subseteq X_i \cap X_j \cap X_k \}.
\]

\begin{claim} \label{claim:Ui}
	The following hold simultaneously with positive probability.
  \begin{enumerate}
	\item \label{nitem:Ui-Ui} $X_i$ has $p n \pm n^{2/3}$ vertices of $G$ for each $i \in V(J)$;
	\item \label{nitem:Ui-Yx} $Y_x$ has $3\alpha n/2 \pm n^{2/3}$ vertices of $J$ for each $x\in V(G)$;
	\item \label{nitem:Ui-Hregular} $\mathcal{H}$ is $(p^6 \alpha \binom{n}{2} \pm 2n^{2 - \rho})$-regular;
	\item \label{nitem:Ui-paths} for each $i \in V(J)$ and each pair of distinct 
	vertices $x, y \in X_i$, there exists~$\ell$ with $1 \leq \ell \leq L$ 
	such that there are at least $\eps^\ell\delta n^\ell/2$ $(x,y)$-paths in $G$
	with exactly $\ell$ inner vertices each, all in $V(G) \setminus X_i$.
\end{enumerate}
\end{claim}

\begin{proofclaim}
	Item~\ref{nitem:Ui-Ui} follows directly from Chernoff's inequality~(\Cref{lemma:chernoff}) as, 
	for each $i \in V(J)$, we have $\expectation[|X_i|] = pn$.  
	For~\ref{nitem:Ui-Yx}, note that $|V(J)| = \lambda (1 + \beta/2)n$ by \ref{item:J-vpartition} and, for any $x\in V(G)$, we have
	\[
	\expectation[|Y_x|] = p|V(J)| = \left(\beta + \frac{\beta^2}{2} \right)n =
	\frac{3\alpha n}{2}.
	\]
	Then \ref{nitem:Ui-Yx} also follows from Chernoff's inequality.
	For~\ref{nitem:Ui-Hregular}, observe that any given edge 
	$\Phi(x_1 y_2, ijk) \in \mathcal{H}'$ survives in $\mathcal{H}$ with probability~$p^6$.  
	Using this, we easily see that $\expectation[\deg_{\mathcal{H}}(x_1 y_2)] 
	= p^6 \deg_{\mathcal{H'}}(x_1 y_2) = p^6 \alpha \binom{n}{2} \pm n^{2 - \rho}$ 
	for each $x_1 y_2 \in E(B)$; a similar calculation holds for
	$\expectation[\deg_{\mathcal{H}}(ijk)]$ for any $ijk \in E(J)$.
	It remains to calculate the expected degree of the edges in $E(Z) \cap V(\mathcal{H})$.  
	Let $x_a i \in E(Z)$.  Conditioning on the event that $x \in X_i$
	(and thus that $x_a i \in V(\mathcal{H})$), each edge in $\mathcal{H}'$ 
	containing $x_a i$ survives with probability $p^5$.  Using this, we obtain that
	\[
	\expectation[\deg_{\mathcal{H}}(x_a i)] 
	= p^5 \deg_{\mathcal{H'}}(x_a i) 
	= p^5 \left( \frac{\alpha \beta}{\lambda} \binom{n}{2} \pm 2n^{2 - \rho} \right) 
	= p^6\alpha \binom{n}{2} \pm 2 n^{2 - \rho}.
	\]
	So~\ref{nitem:Ui-Hregular} follows from Chernoff's inequality.
	
	In the remainder of the proof we use McDiarmid's
        inequality~(\cref{lemma:mcdiarmid}) to check
        that~\ref{nitem:Ui-paths} holds.  Given $i \in V(J)$, since
        $G$ is $(\delta,L)$-robustly-connected, for each pair of
        distinct vertices $x, y \in X_i$, there exists $\ell$ with $1
        \leq \ell \leq L$ such that there are at least $\delta n^\ell$
        $(x,y)$-paths in $G$ with exactly $\ell$ inner vertices each.  Hence,
        the expected number of such paths with all internal vertices
        in $V(G)\setminus X_i$ is at least $(1-p)^{\ell}\delta
        n^{\ell}$. Since the removal or addition of a vertex in $X_i$
        changes the number of $(x,y)$-paths by $O(n^{\ell-1})$
        paths, one can check by using McDiarmid's inequality that, for
        a given $i\in V(J)$ and a pair $x, y \in X_i$, the probability
        that we have less than $\eps^\ell\delta n^\ell/2$
        $(x,y)$-paths with exactly $\ell$ inner vertices, all in $V(G)
        \setminus X_i$, is $\exp(-\Omega(n))$. Since
        $|V(J)|=O(n)$ and there are $O(n^2)$ possible pairs $x,y\in
        X_i$, item~\ref{nitem:Ui-paths} follows from the union bound.
\end{proofclaim}

From now on, we assume that the sets $\{ X_i : i \in V(J)\}$, $\{ Y_x : x \in V(G)\}$,
and the hypergraph~$\calH$ satisfy properties~\ref{nitem:Ui-Ui}--\ref{nitem:Ui-paths} 
of Claim~\ref{claim:Ui}.

\medskip

\noindent \emph{Step 2: Verifying properties of $\mathcal{H}$.}
We start by defining
\(
d := \Delta_1(\mathcal{H}).
\)
Note that from Claim~\ref{claim:Ui}\ref{nitem:Ui-Hregular}, we have
\(
d = p^6 \alpha \binom{n}{2} \pm 2n^{2 - \rho}.
\)
We will apply \Cref{theorem:hypergraphmatching} to~$\mathcal{H}$.
The following claim guarantees that $\calH$ satisfies the required hypotheses.

\begin{claim} \label{claim:hypothesesH}
	The following facts about $\calH$ hold.
	\begin{enumerate}[\upshape{\textbf{(H\arabic*)}},labelindent=0pt,labelwidth=\widthof{\ref{Hlast}},leftmargin=!]
		\item \label{item:h-vertices} $\mathcal{H}$ has at most $\exp(d^{\rho^3})$ vertices;
		\item \label{item:h-delta_1} $d(1 - d^{-\rho/3}) \leq \delta_1(\mathcal{H}) \leq \Delta_1(\mathcal{H}) = d$; and
		\item $\Delta_2(\mathcal{H}) \leq d^{2/3}$.\label{Hlast}
	\end{enumerate}
\end{claim}

\begin{proofclaim}
	Item~\ref{item:h-vertices} follows from the fact that $|V(\calH)|
	= |E(B)|+|E(J)|+|E(Z)| \leq \alpha n^2 + \alpha n^2 +
	2n(\lambda+\lambda\beta/2)n \leq \exp(d^{\rho^3})$, where the last
	inequality holds with a lot of room to spare.
	
	For~\ref{item:h-delta_1}, the upper bound follows from the
	definition of~$d$ and the lower bound follows from $\delta_1(\calH)
	\geq p^6 \alpha \binom{n}{2} - 2n^{2 - \rho} \geq d - d^{1 - \rho/3}$.
	
	It remains to verify that $\Delta_2(\mathcal{H}) \leq
	d^{2/3}$. This will require some work. First, note that each edge
	in $\mathcal{H}$ is of the form $\Phi(x_1 y_2, ijk)$ for $x_1 \in
	V_1, y_2 \in V_2$, and $ijk \in V(J)$.  We need to select two
	distinct vertices $e, f$ in $V(\mathcal{H})$ and calculate $\deg_\calH(e,f)$.  A vertex $e$ of~$\mathcal{H}$ can belong to $E(B)$,
	$E(J)$, or $E(Z)$.  We consider all the six possible combinations
	for $e, f$.
	
	Let $e, f \in V(\mathcal{H})$.  Since each edge of $\mathcal{H}$
	is of type $\Phi(x_1 y_2, ijk)$ for $x_1 y_2 \in E(B)$ and $ijk
	\in E(J)$, and each of these contains exactly one vertex in $E(B)$
	and one vertex in $E(J)$, for $e,f\in E(B)$ or $e,f \in E(J)$, we
	have $\deg_{\mathcal{H}}(e, f) = 0$.  Furthermore, since $x_1 y_2
	\in E(B)$ and $ijk \in E(J)$ completely determine the edge
	$\Phi(x_1 y_2, ijk)$, if $e \in E(B)$ and $f \in E(J)$, then
	we have $\deg_{\mathcal{H}}(e,f) \leq 1$.
	
	In view of the above discussion, we may assume that $e \in E(Z)$,
	and without loss of generality we assume $e=x_1i$ for some $x_1
	\in V_1$, and $i \in V(J)$.  There are now three cases to
	consider, depending whether $f$ belongs to $E(B)$, $E(J)$, or
	$E(Z)$.
	
	Suppose first that $f = x_1 y_2 \in E(B)$. We will count the
	number of pairs $\{j,k\}$ such that $\Phi(x_1 y_2, ijk)$ is an edge of
	$\calH$.  In particular, it must happen that $jk \in N_J(i)$,
	thus $\deg_{\mathcal{H}}(e, f) \leq \deg_{J}(i) \leq n \leq
	d^{2/3}$.  Similarly, if $f = ijk \in E(J)$, then we count the
	number of vertices $y_2\in V_2$ such that $\Phi(x_1 y_2, ijk)$ is an edge of
	$\calH$, which is at most $\deg_B(x_1) \leq n \leq d^{2/3}$.
	
	Finally, suppose that $f \in E(Z)$ and recall that $e=x_1i$. If
	$f=y_2j$ with $y_2 \in V_2$ and $j \in V(J)$, then $\deg_{H}(e,
	f)$ is the number of edges of $J$ containing $i$ and $j$. In the
	worst case, $i = j$, we have $\deg_{H}(e, f) = \deg_{J}(i) \leq n
	\leq d^{2/3}$. On the other hand, if $f=x_1 j$ with $j \in V(J)$,
	we have $j \neq i$ and then to estimate $\deg_{H}(e, f)$ we need
	to count the number of $y_2 \in N_B(x_1) \subseteq V_2$ and the
	number of $k \in N_J(ij)$. The number of choices for $y_2$ is at
	most $n$ and the number of choices for $k$ is at most $\Delta_2(J)
	\leq \log^2 n$. Therefore, $\deg_{H}(e, f) \leq n \log^2 n \leq
	d^{2/3}$, as required.
\end{proofclaim}

\medskip

\noindent \emph{Step 3: Setting the conflicts.}
We must ensure that the collection $\calQ$ of $2$-matchings we want
to obtain is $\eps'$-compact: each 2-matching in $\calQ$ has at most 
$\eps' n$ paths and each cycle in $\calQ$ has length at least $1/\eps'$. 
This condition on the cycle lengths will be encoded by using conflicts.

Recall that $D$ is the oriented graph obtained by orienting each
edge of $G$ uniformly at random. In what follows, let
\(
r:= 1/\eps'.
\)
We define our conflict hypergraph $\mathcal{C}$ on vertex set $E(\mathcal{H})$ 
and edge set defined as follows. For each $\ell$ with $3 \leq \ell \leq r$, 
each $\ell$-length directed cycle $C \subseteq D$ with vertices 
$\{v^1, \dotsc, v^\ell\}$, each $i \in U_1$, and each $j_1 k_1, \dotsc, j_\ell k_\ell \in N_J(i)$, 
we define the following edge:
\[
F(C,i,j_1k_1,\dotsc,j_\ell k_\ell) = \{\Phi(v^{a}_1 v^{a+1}_2, i\,j_a k_a) : 1 \leq a \leq \ell\},
\]
where $v^{\ell+1}_2 = v^1_2$. 
Note that $F(C, i, j_1 k_1, \dotsc, j_\ell k_\ell)$ corresponds to a set of $\ell$ edges of~$\calH'$,
associated to the triples $(i\,j_1 k_1), \dotsc, (i\,j_\ell k_\ell)$ 
and the edges of the $\ell$-length directed cycle $C$ in $D$. 
In such a case, we say $i$ is the \emph{monochromatic colour} of the conflicting cycle $C$. 
The edges of the conflict hypergraph $\mathcal{C}$ consist
of all edges of type $F(C, i, j_1 k_1, \dotsc, j_\ell k_\ell)$ which are contained in $E(\mathcal{H})$. 
The next claim establishes 
that $\calC$ is a $(d, r, \rho)$-bounded conflict system for $\calH$. 

\begin{claim} \label{claim:hypothesesC}
	The following facts about $\calC$ hold.
	\begin{enumerate}[\upshape{\textbf{(C\arabic*)}},labelindent=0pt,labelwidth=\widthof{\ref{item:Cdegs}},leftmargin=!]
		\item \label{item:size} $3 \leq |F| \leq r$ for each $F \in \mathcal{C}$; and
		\item \label{item:Cdegs} $\Delta_{j'}(\mathcal{C}^{(j)}) \leq r d^{j - j' - \rho}$ for each $3 \leq j \leq r$ and $1 \leq j' < j$. 
	\end{enumerate}
\end{claim}

\begin{proofclaim}
	Fact~\ref{item:size} is immediate from the construction of $\calC$. 
	
	In order to prove~\ref{item:Cdegs}, 
	fix $j$ and $j'$ with $3 \leq j \leq r$ and $1 \leq j' < j$.  
	To prove that $\Delta_{j'}(\calC^{(j)}) \leq r d^{j - j' - \rho}$, 
	we need to show that any set of $j'$ edges of~$\calH$ is contained 
	in at most $r d^{j - j' - \rho}$ conflicts of size $j$ in~$\calC$. 
	Let $\calR$ be any set of $j'$ edges in $\calH$, 
	say $\calR = \{ \Phi( x_1^a y_2^a, i^a j^a k^a ) : 1 \leq a \leq j' \}$.  
	We want to bound the degree of~$\calR$ in $\calC^{(j)}$. 
	Each conflict of size $j$ is defined by a length-$j$ directed cycle 
	in~$D$, the monochromatic colour of the conflict, and a corresponding choice of labels for each edge in the cycle;
	now we estimate the number of valid choices for each of these three elements.
	
	We begin by estimating how many possibilities there are for choosing a suitable cycle.
	Note that $R := \{ (x^a, y^a) : 1 \leq a \leq j' \}$ is a set of at most $j'$ edges of~$D$.  
	If~$|R|<j'$, then there are repeated edges from $D$ in $R$, 
	and in this case the degree of~$\calR$ in $\calC^{(j)}$ is zero. 
	So we can assume that $|R|=j'$ and, by~\Cref{lemma:countingcycles}, 
	there are at most $j^{j'} n^{j-j'-1}$ length-$j$ directed cycles in $D$ which contain~$R$.
	
	Now we consider the possible choices for the monochromatic colour $i$ of the conflicting cycle.
	Note that if there is no common $i \in V(J)$ among all labels $i^a j^a k^a$ for ${1 \leq a \leq j'}$, the degree of~$\calR$ in $\calC^{(j)}$ is zero, because there is no available `monochromatic colour' at all.
	This also implies that there are at most $3$ possible choices for $i$, because it must be one of the three labels which belong to $i^1 j^1 k^1$, say.
	
	Having fixed a directed cycle $C$ which contains $R$, and a monochromatic colour $i$ for the conflict,
	now we count the number of labels associated with each edge of $C$.
	For edges of~$\calR$, the choices are already given, 
	and for the remaining ${j-j'}$ edges of $C$ not in~$\calR$, the labels must be chosen among the neighbours of $i$ in the hypergraph $J$.
	Since $i$ has $\beta n / \lambda \pm n^{2/3}\leq n$ neighbours in $J$ by~\ref{item:J-regular}, 
	in this step we have at most $n^{j-j'}$ possible choices. 
	%
	Therefore,
	\[
	\Delta_{j'}(\mathcal{C}^{(j)}) \leq j^{j'} n^{j - j' - 1}
	\cdot 3 \cdot n^{j-j'} = 3 j^{j'} n^{2(j-j') - 1} \leq r d^{j - j' - \rho},
	\]
	where in the last step we used that $d = \Theta(n^2)$ and $n$ is sufficiently large.
\end{proofclaim}

\medskip

\noindent \emph{Step 4: Setting the test sets.}
For each $x_1 \in V_1 \subseteq V(B)$ and each $y_2 \in V_2 \subseteq V(B)$, define
$Z_{x_1} = \{ \Phi(x_1 y_2, ijk ) \in E(\mathcal{H}) : y_2 \in N_B(x_1), ijk \in E(J)\}$ 
and define $Z_{y_2}$ in a similar manner.  Furthermore, define 
$Z_i = \{ \Phi(x_1 y_2, ijk ) \in E(\mathcal{H}) : x, y \in X_i\}$ for each $i \in V(J)$. 
We claim that $\mathcal{Z} := \{Z_{x_1} : x_1 \in V_1 \} \cup \{ Z_{y_2} : y_2 \in V_2 \} 
\cup \{Z_{i} : i \in V(J)\}$ is a suitable family of trackable sets.
Specifically, the next claim shows that $\calZ$ is not very large and has only $(d,\rho)$-trackable sets.

\begin{claim} \label{claim:trackable}
	The following facts about $\calZ$ hold.
	\begin{enumerate}[\upshape{\textbf{(Z\arabic*)}},labelindent=0pt,labelwidth=\widthof{\ref{item:Zlast}},leftmargin=!]
		\item $|\mathcal{Z}| \leq \exp(d^{\rho^3})$; and
		\item \label{item:Zlast} each $Z \in \mathcal{Z}$ is $(d,\rho)$-trackable.
	\end{enumerate}
\end{claim}

\begin{proofclaim}
	Because $| \mathcal{Z} | = |V(B)| + |V(J)| \leq 3n$ and
	$d=\Theta(n^2)$, we have ${|\mathcal{Z}| \leq \exp(d^{\rho^3})}$.  
	It remains to check \ref{item:Zlast}, which means to prove that $|Z_v| \geq d^{1+\rho}$ for each $v \in V_1 \cup V_2$, and $|Z_i| \geq d^{1+\rho}$ for each $i \in V(J)$.
	
	First, suppose $v = x_1 \in V_1$.  From~\ref{item:h-delta_1} and 
	$\delta(B) \geq \alpha n/2 - 2n^{1 - \rho} \geq \alpha n/3$, we have
	\begin{align*}
		|Z_{x_1}| & = \sum_{y_2 \in N_B(x_1)} \deg_{\mathcal{H}}(x_1 y_2)
		\geq \delta_1(\mathcal{H}) |N_B(x_1)|
		\geq d(1-d^{-\rho/3}) \alpha n / 3 \geq d^{1+\rho},
	\end{align*}
	where in the last step we used that $d = \Theta(n^2)$ and that $n$ is large.  
	The calculations for $v = y_2 \in V_2$ are identical. 
	Next, we note that for any $i\in V(J)$ we have
	\begin{align*}
		\qquad |Z_{i}| & = \sum_{x \in X_i} \deg_{\mathcal{H}}(x_1 i)
		\geq \delta_1(\mathcal{H}) |X_i|
		\geq d(1-d^{-\rho/3})(pn-n^{2/3})
		\geq d^{1+\rho}. \qedhere
	\end{align*}
\end{proofclaim}

\medskip

\noindent \emph{Step 5: Finishing the proof.}
Recall that $d=\Delta_1(\calH)$. By Claims~\ref{claim:hypothesesH},
\ref{claim:hypothesesC} and~\ref{claim:trackable}, we can apply
\Cref{theorem:hypergraphmatching} to~$\mathcal{H}$, using
$\mathcal{C}$ as a conflict system and $\mathcal{Z}$ as a set of
trackable sets, and $\rho/3$ in place of $\rho$. By doing so, we obtain a matching
$\mathcal{M} \subseteq \mathcal{H}$ such that
\begin{enumerate}
	\item $\mathcal{M}$ is $\mathcal{C}$-free,
	\item \label{item:matching-size}$\mathcal{M}$ has size at least $(1
	- d^{- (\rho/3)^3})|V(\mathcal{H})|/8$, and
	\item \label{item:matching-Zz}$|Z_{a} \cap \mathcal{M}| = (1 \pm d^{-
		(\rho/3)^3})|\mathcal{M}||Z_a|/|E(\mathcal{H})|$ for each $a
	\in V(B)\cup V(J)$.
\end{enumerate}

Recall that $\lambda = \beta/(1-\eps)$ and let $t = \lambda n$.
Using $\mathcal{M}$, we define the graphs $\{ Q_i \}_{i=1}^t$ as follows.  
For an edge $x_1 y_2 \in E(B)$, suppose there exists $ijk \in E(J)$ 
such that $\Phi(x_1 y_2, ijk) \in \mathcal{M}$.  
In that case, we will add the edge $xy \in E(G)$ to the graph $Q_a$ 
such that $a \in \{i, j, k \} \cap U_1$.  
To finish, we verify that $\calQ = \{Q_i\}_{i=1}^t$ is an $(\eps^\ell\delta/2, L, \lambda, \eps')$-separator for $G$,
which means we need to show that $\calQ$ is a collection of $2$-matchings in~$G$ that satisfies \ref{item:firstpacking-connecting}--\ref{item:firstpacking-almostperfect} with $\eps^\ell\delta/2$, $\lambda$ and $\eps'$ in the place of $\delta$, $\beta$ and $\eps$, respectively.

We start by verifying that $\calQ$ is a collection of 2-matchings.  Note that, 
for each $1 \leq i \leq t$, the graph $Q_i$ has maximum degree at most $2$.  
Indeed, let $x \in V(G)$ be any vertex.  
Since $\mathcal{M}$ is a matching in $\calH$, at most two edges in $\mathcal{M}$ 
can cover the vertices $x_1 i, x_2 i \in V(\mathcal{H})$; and this 
will yield at most two edges adjacent to $x$ belonging to $Q_i$.

Now we verify that \ref{item:firstpacking-connecting} holds.
First, we check that $\calQ$ is $\eps'$-compact, that is, each 2-matching in~$\calQ$ 
has at most $\eps' n$ paths and each cycle in~$\calQ$ has length at least $1/\eps'$. 
The latter holds because we avoided the conflicts in $\mathcal{C}$.  
More precisely, an $\ell$-cycle in~$Q_i$ corresponds to a sequence of $\ell$ edges, 
all of which are in~$Q_i$.  This means the cycle was formed from a length-$\ell$ 
directed cycle in~$D$, all of whose edges were joined (via~$\mathcal{M}$) to triples 
in $J$, all containing vertex $i \in V(J)$.  Recall that $r = 1/\eps'$.
If $\ell \leq r$, this forms a conflict in $\mathcal{C}$, 
so, as $\mathcal{M}$ is $\mathcal{C}$-free, we deduce that $\ell > r$.
To check that $Q_i$ has few paths, first observe that $V(Q_i) \subseteq X_i$ for each $i \in V(J)$.  
Indeed, if $xy \in E(Q_i)$, then we have that, say, $(x,y) \in E(D)$ and 
$\Phi(x_1 y_2,ijk)$ is an edge in $\mathcal{M}$ for some $j, k$.  
But, since $\mathcal{M} \subseteq \mathcal{H}$, by the definition of $\mathcal{H}$,
we have $x, y \subseteq X_i$, as required.
Now, from \ref{item:matching-size}, \ref{item:matching-Zz}, the fact that 
$\calH$ is an 8-graph close to $d$-regular, and $|Z_i| \leq d|X_i|$, we have that 
$|E(Q_i)| = |Z_i \cap \mathcal{M}| \geq (1 - \eps'/2)|X_i| \geq (1 - \eps'/2)|V(Q_i)|$,
and then the number of degree-one vertices in $Q_i$ is at most 
$2(|V(Q_i)|-|E(Q_i)|) \leq \eps' |V(Q_i)| \leq \eps' n$.
To see the second part of \ref{item:firstpacking-connecting}, 
we need to show that $\mathcal{Q}$ is $(\eps^\ell\delta/2, L)$-robustly-connected.
Because $V(Q_i) \subseteq X_i$, we deduce from Claim~\ref{claim:Ui}\ref{nitem:Ui-paths} 
that $Q_i$ is $(\eps^\ell\delta/2, L)$-robustly-connected, as required.
We conclude that \ref{item:firstpacking-connecting} holds.

We have already stated that $|\mathcal{Q}| = \lambda n$, so the first part of \ref{item:firstpacking-separating} holds.
The second part of~\ref{item:firstpacking-separating} can be checked as follows.
Let $e, f$ be distinct edges of $E(\calQ)$. Thus, there are orientations
$(x,y), (x',y') \in E(D)$ of $e, f$ respectively, and edges $ijk, i'j'k' \in E(J)$ 
such that $\Phi(x_1 y_2, ijk)$ and $\Phi(x'_1 y'_2, i'j'k')$ belong to $\mathcal{M}$.  
We have, respectively, that $A_e := \{ a : e \in E(Q_a) \} = \{i, j, k\} \cap U_1$
and $A_f := \{ a : f \in E(Q_a) \} = \{i', j', k'\} \cap U_1$.
For a contradiction, suppose $A_e \subseteq A_f$. If $|A_e| = 3$, then 
we would have that $ijk = i'j'k'$, contradicting that $\mathcal{M}$ 
is a matching, so $|A_e| = 2$; say, $A_e = \{i, j\}$, and $ij$ is a
pair in~$V_1$ (from the construction of $J$).  We recall that
by~\ref{item:J-antichain} no pair $ij$ is contained both in an edge
with intersection $2$ and $3$ with $V_1$, so this rules out the case
$|A_f| = 3$.  Thus we can only have $A_e = A_f = \{i,j\}$.  But
again \ref{item:J-antichain} implies $ij$ is contained in a unique
edge in $J$, say, $ijr$.  This implies that $ijk = i'j'k' = ijr$,
contradicting the fact that $\mathcal{M}$ is a matching.
Therefore $\calQ$ strongly separates $E(\calQ)$, and~\ref{item:firstpacking-separating} holds.

To prove \ref{item:firstpacking-endpoints}, let $x\in V(G)$. Recall 
that $Y_x\subseteq V(J)$ is the random set $Y_x = \{i : x\in X_i\}$ and 
from Claim~\ref{claim:Ui}\ref{nitem:Ui-Yx} we have that $|Y_x| = 3\alpha n/2 \pm n^{2/3}$. 
Note that if $x$ is the end of a path in some $2$-matching $Q_i$, 
then there is an edge $\Phi(x_1 y_2, ijk)$ in $Z_{x_1} \cap \calM$, 
but no edge $\Phi(x_2y_1, ijk)$ is in $Z_{x_2} \cap \calM$; 
or there is an edge $\Phi(x_2 y_1, ijk)$ in $Z_{x_2} \cap \calM$, 
but no edge $\Phi(x_1y_2, ijk)$ is in $Z_{x_1} \cap \calM$.  
This motivates the following definition: for each $x\in V(G)$, 
a set $F(x_1)$ of indexes $i\in V(J)$ such that there is an
edge $\Phi(x_1y_2,ijk)$ in $\calM$; and a set $F(x_2)$ of indexes
$i\in V(J)$ such that there is $\Phi(x_2y_1,ijk)$ in $\calM$.  
Note that, from the way we construct $\calH$, we know that 
$F(x_1),F(x_2) \subseteq Y_x$.  In view of the above discussion, 
the number of times $x$ is the endpoint of a path in the $2$-matchings 
of $\calQ$ is the number of indexes $i\in Y_x$ such that 
$i\notin F(x_1)\cap F(x_2)$. Therefore, this number of indexes $i$ 
such that $x$ is the endpoint of a path in $Q_i$ is at most
\begin{align*}
|Y_x\setminus F(x_1)|+|Y_x\setminus F(x_2)|
&\leq 3\alpha n + 2n^{2/3} - 3(|Z_{x_1}\cap\calM| +
|Z_{x_2}\cap\calM|)\\
&\leq 3\alpha n + 2n^{2/3} - 3(\deg_G(x)-\eps'n/2)\\
&\leq 3\alpha n + 2n^{2/3} - 3(\alpha n - n^{1-\rho}-\eps'n/2)\\
&\leq \eps'n,
\end{align*}
where in the first inequality we use the facts that $|F(x_1)| =
3|Z_{x_1}\cap\calM|$ and $|F(x_2)| = 3|Z_{x_2}\cap \calM|$, and also
that $Y_x \leq 3\alpha n + n^{2/3}$; in the second inequality we
use~\eqref{eq:Zx1Zx2}; the third inequality follows from
$\deg_G(x)\geq \alpha n - n^{1-\rho}$; and since $n$ is sufficiently
large, the last inequality holds. This
verifies~\ref{item:firstpacking-endpoints}.

To see \ref{item:firstpacking-sizes}, let $e \in E(\mathcal{Q})$ be arbitrary.
As explained before, there exists an orientation $(x,y) \in E(D)$ of $e$ and an edge $ijk \in E(J)$ such that $\Phi(x_1 y_2, ijk)$ belongs to $\calM$, and $\{ a: e \in E(Q_a) \} = \{i, j, k\} \cap U_1$.
Since the latter set obviously has at most three elements, \ref{item:firstpacking-sizes} follows. 

Finally, property \ref{item:firstpacking-almostperfect} follows from the
properties of the chosen test sets.  More precisely, we want to prove 
that $\Delta(G') \leq \eps'n$ for $G':=G-E(\calQ)$. Since for
any $x\in V(G)$ we have $\deg_{G'}(x) = \deg_G(x)- (|Z_{x_1} \cap
\mathcal{M}| + |Z_{x_2} \cap \mathcal{M}|)$, it is enough to prove
that $|Z_{x_1} \cap \mathcal{M}| + |Z_{x_2} \cap \mathcal{M}| \geq
\deg_G(x) - \eps'n$. For that, by using~\ref{item:matching-size}
and~\ref{item:matching-Zz} and the facts that $|E(\calH)| \leq
|V(\calH)|\Delta_1(\calH)/8$ and $|Z_{x_1}| +|Z_{x_2}| \geq
\delta_1(\calH)(|N_B(x_1)|+|N_B(x_2)|) \geq d(1-d^{-\rho/3})\deg_G(x)$, 
we have the following for any $x\in V(G)$:
\begin{align}
	|Z_{x_1} \cap \mathcal{M}| + |Z_{x_2} \cap \mathcal{M}|
	&\geq \frac{(1-d^{-(\rho/3)^3})^2 (|Z_{x_1}|+|Z_{x_2}|)|V(\calH)|/8}{|V(\calH)|\Delta_1(\calH)/8}\nonumber\\
&\geq (1-d^{-(\rho/3)^3})^2 (1-d^{-\rho/3})\deg_G(x)\nonumber\\
&\geq \deg_G(x) - \eps'n/2, \label{eq:Zx1Zx2}
\end{align}
where inequality~\eqref{eq:Zx1Zx2} holds for sufficiently large $n$ 
because $\deg_G(x)=\Theta(n)$ and $d=\Theta(n^2)$. Then, we verified
that~\ref{item:firstpacking-almostperfect} holds.
This finishes the proof of the lemma.
\end{proof}

\section{Breaking cycles and connecting paths} \label{section:breakingbad}
\label{sec:connecting}

For a real number $\eps \geq 0$, a collection $\calP$ of paths in $G$ is an 
\emph{$\eps$-almost separating path system} if there exists a set $E' \subseteq E(G)$ 
such that $\calP$ separates every edge in~$E'$ from all other edges in $G$ 
and~$\Delta(G-E') \leq \eps n$.
Note that such $\calP$ strongly separates~$E'$. 

A 2-matching that has no cycle is called \emph{acyclic}. A collection
$\calQ$ of 2-matchings is \emph{acyclic} if each 2-matching in $\calQ$
is acyclic.
The next lemma shows that a collection of 2-matchings as in the output
of \Cref{lemma:firstpacking} (that is, a separator, as in \Cref{definition:separator}) can be converted into an acyclic
$2$-matching with only a very small loss in its properties.

\begin{lemma}\label{lem:acyclic}
	Let $1/n \ll \delta, L, \beta, \eps$.
	If $G$ is an $n$-vertex graph and there exists a $(\delta, L,\beta,\eps)$-separator for $G$,
  then there also exists an acyclic $(\delta, L,\beta,5\eps)$-separator for $G$.
\end{lemma}

\begin{proof}
	Let $\mathcal{Q}$ be a $(\delta, L,\beta,\eps)$-separator for $G$.
  We will remove one edge from each cycle in the $2$-matchings of
  $\calQ$. First note that, after removing one edge from each cycle of
  an $\eps$-compact $2$-matching $Q$, we obtain an acyclic
  $2\eps$-compact $2$-matching, because the maximum number of cycles
  in $Q$ is $\eps n$. Since $Q$ is $(\delta,L)$-robustly-connected,
  then it remains so after the removal of such edges. Thus, the
  collection of $2$-matchings obtained after the removal of these
  edges from $\calQ$ satifies~\ref{item:firstpacking-connecting} with
  $2\eps$ in the place of $\eps$.  Note that such collection also
  satifies~\ref{item:firstpacking-separating}
  and~\ref{item:firstpacking-sizes}.  Moreover, if we remove from the
  $2$-matchings at most $4\eps n$ edges incident to each vertex of
  $G$, then~\ref{item:firstpacking-endpoints} will hold with $5\eps n$
  in the place of $\eps n$.  Moreover, the degree of $u$ in $G-E_i$
  will increase by at most $4\eps n$, which implies
  condition~\ref{item:firstpacking-almostperfect} with $5\eps$ in
  the place of $\eps n$.
  
  In view of the previous discussion, our goal is to prove that there
  is a set of edges with at most $4\eps n$ edges incident to each
  vertex whose removal makes $\calQ$ acyclic.
  Let $C_1,\ldots,C_T$ be the cycles in 2-matchings of $\calQ$ and
  note that $T \leq \eps \beta n^2$. 
  For $1 \leq i \leq T$, let $\calX_i$ be the edges of the cycle $C_i$, 
  and let $X_i$ be an edge chosen uniformly at random from $\calX_i$.
  
  Let $S$ be the edge set $\{X_1,\ldots,X_T\}$ and, for each vertex
  $u$, let $f^u(X_1,\ldots,X_T)$ be the degree of $u$ in $G[S]$.  Note
  that, since $u$ is in at most one cycle $C_i$ of a $2$-matching of
  $\calQ$ and each cycle has length at least $1/\eps$, the edge $X_i$
  was chosen as one of the two edges incident to $u$ with probability
  at most $2\eps$. Then, because the number of $2$-matchings is at most
  $\beta n$, we have that $\expectation[f^u(X_1,\ldots,X_T)] \leq
  2\eps\beta n$.
  
  Let $(x_1,\ldots,x_T)$ and $(x'_1,\ldots,x'_T)$ be in $\calX_1
  \times \dotsb \times \calX_T$, differing in exactly one coordinate,
  that is, $x_j = x'_j$ for every $j \in \{1,\ldots,T\}$ with $j \neq
  i$.  Note that $f$ is such that $|f^u(x_1,\ldots,x_T) -
  f^u(x'_1,\ldots,x'_T)| \leq 1$.  In fact, $f^u(x_1,\ldots,x_T) =
  f^u(x'_1,\ldots,x'_T)$ if $u$ is not in $C_i$.  So we can set $c_i =
  1$ if $u$ is in $C_i$ and $c_i = 0$ otherwise.  As $u$ is in at most
  $\beta n$ of the $T$ cycles, we have that $\sum_j c_j^2 \leq \beta
  n$.  By using McDiarmid's inequality~(\Cref{lemma:mcdiarmid}), we
  obtain
  \[
    \Pr(f^u(X_1,\ldots,X_T) \geq 4\eps\beta n) \ \leq \
    \exp\left(-\frac{8\eps^2\beta^2n^2}{\beta n}\right) \ \leq \
    \exp\big(-8\eps^2\beta n\big).
  \]
  Thus, by the union bound, the probability that the maximum degree in
  $G[S]$ is less than~$4\eps\beta n$ is at least $1 - n\cdot
  \exp\big(-8\eps^2\beta n\big)$, which, for large enough $n$, is
  positive.  This means there is a choice of edges whose removal from
  the cycles in the 2-matchings of $\calQ$ makes~$\calQ$ acyclic and
  satisfying~\ref{item:firstpacking-connecting}--\ref{item:firstpacking-almostperfect},
  with $5\eps$ in the place of $\eps$.
\end{proof}

The next lemma is the main result of this section, 
and will be used in Section~\ref{sec:mainresult} to prove our main result. 

\begin{lemma} \label{lemma:connectingpaths} For each $\eps$, $\delta$,
  and $L$, there exist $\eps'$ and $n_0$ such that the following holds
  for every $n$-vertex graph $G$ with~$n \geq n_0$ and every $\beta
  \in (0,1)$.  If $\calQ$ is a $(\delta, L,\beta,\eps')$-separator
  for~$G$, then there exists an $\eps$-almost separating path system
  with $\beta n$ paths.
\end{lemma}

\begin{proof} 
  Let $t = \beta n$.  We can apply Lemma~\ref{lem:acyclic} to transform $\calQ$ 
  into an acyclic collection of $2$-matchings, adjusting the value of $\eps'$ accordingly. 
  Let $\calQ = \{Q_1, \dotsc, Q_t\}$. 
  
  We will describe a sequence $\calC_0,\ldots,\calC_t$ of collections of acyclic 2-matchings in~$G$ 
  and sets $E_0,\ldots,E_t$ of edges of~$G$, the idea being that $\mathcal{C}_i$ strongly separates $E_i$, 
  and that each $\mathcal{C}_i$ will be obtained from $\mathcal{C}_{i-1}$ by replacing $Q_i$ with a path $P_i$. 
  Then $\mathcal{C}_t$ will be the desired path system.
  
  For each vertex $u$ and $0 \leq i \leq t$, let~$d_i(u)$ be the total number of paths 
  in the 2-matchings in~$Q_{i+1},\ldots,Q_t$ that have~$u$ as an endpoint. 
  We will make sure the following invariants on $\calC_i$ and~$E_i$ hold for each $0 \leq i \leq t$: 
  \begin{enumerate}[\upshape{\textbf{(I\arabic*)}},labelindent=0pt,labelwidth=\widthof{\ref{inv-Delta}},leftmargin=!]
  	\item each $\calC_i$ separates every edge in $E_i$ from all other edges of~$G$;\label{inv-sep}
  	\item edges in more than three of the 2-matchings in~$\calC_i$ are 
  	in $E(\calC_i) \setminus E(\calC_0)$; and\label{inv-3x}
  	\item the degree of each vertex $u$ in $G-E_i$ is at most $\eps n$ if $d_i(u) = 0$ 
  	and at most $\sqrt{\eps'}n - 2d_i(u)$ if $d_i(u) > 0$.\label{inv-Delta}
  \end{enumerate}
  
  Let $E_0 = E(\calQ)$ and $\calC_0 = \calQ$.  
  Note that $\calC_0$ and $E_0$ satisfy the three invariants.
  We will define $\calC_i = (\calC_{i-1} \setminus \{Q_i\}) \cup \{P_i\}$ 
  for $i=1,\ldots,t$, where~$P_i$ is a path that contains all paths in $Q_i$.  
  Therefore, if invariants~\ref{inv-sep} and~\ref{inv-Delta} hold for $i = t$ 
  and $\eps' \leq \eps^2$, then~$\calC_t$ will be an $\eps$-almost separating path 
  system with $t$ paths, and the proof of the lemma will be complete, as $t = \beta n$. 
  
  Suppose $i \geq 1$. To describe how we build $P_i$ from $Q_i$, we need some definitions. 
  Let~$f$ be an edge of~$G$ not in~$Q_i$ such that~$Q_i+f$ is a~2-matching. 
  Let~$E^f$ be $f$ plus the set of edges of $Q_i$ in~$E_{i-1}$ that are not 
  separated from $f$ by~$(\calC_{i-1} \setminus \{Q_i\}) \cup \{Q_i+f\}$. 

\begin{claim}\label{clm:atmost3}
	If $f$ is in at most three of the 2-matchings in $\calC_{i-1}$, then $|E^f| \leq 4$.
\end{claim}

\begin{proof}
	Suppose there are three edges $a$, $b$, and $c$ in $E(Q_i) \cap E_{i-1}$ 
	that are not separated from $f$ by~$(\calC_{i-1} \setminus \{Q_i\}) \cup \{Q_i+f\}$. 
	By invariant~\ref{inv-sep}, there are 2-matchings $Q_{a\bar{b}}$, $Q_{b\bar{c}}$, and $Q_{c\bar{a}}$ 
	in~$\calC_{i-1}$ such that $a \in E(Q_{a\bar{b}})$ but $b \not\in E(Q_{a\bar{b}})$, $b \in E(Q_{b\bar{c}})$ 
	but $c \not\in E(Q_{b\bar{c}})$, and~$c \in E(Q_{\bar{a}c})$ but $a \not\in E(Q_{\bar{a}c})$. 
	Clearly these three 2-matchings are distinct and are not $Q_i$, because $a$, $b$, and $c$ are in~$Q_i$. 
	So they are in $(\calC_{i-1} \setminus \{Q_i\}) \cup \{Q_i+f\}$ and they must contain $f$
	because $a$, $b$, and $c$ are not separated from $f$ in $(\calC_{i-1} \setminus \{Q_i\}) \cup \{Q_i+f\}$. 
	By the hypothesis of the claim, these are the only 2-matchings in~$\calC_{i-1}$ containing $f$. 
	Hence, repeating the argument for $a$, $b$, $c$ in the inverse order, 
	we deduce that either $a \in Q_{b\bar{c}}$,~$b \in Q_{\bar{a}c}$, and $c \in Q_{a\bar{b}}$, 
	or $a \not\in Q_{b\bar{c}}$, $b \not\in Q_{\bar{a}c}$, and $c \not\in Q_{a\bar{b}}$.
	
	Now, suppose there is a fourth edge $d$ in $E(Q_i) \cap E_{i-1}$ not 
	separated from~$f$ by the collection~$(\calC_{i-1} \setminus \{Q_i\}) \cup \{Q_i+f\}$. 
	Consider the former of the two cases above and, for clarity, rename the 
	three 2-matchings to $Q_{ab\bar{c}}$, $Q_{\bar{a}bc}$, and $Q_{a\bar{b}c}$. 
	Then $d$ must be in $Q_{\bar{a}bc}$, to be separated from $a$, 
	and $d$ must be in $Q_{a\bar{b}c}$, to be separated from $b$. 
	But now there is no way to separate $c$ from $d$, a contradiction. 
	The other case is analagous.  Indeed, for clarity, rename the three 2-matchings 
	to $Q_{\bar{a}b\bar{c}}$, $Q_{\bar{a}\bar{b}c}$, and $Q_{a\bar{b}\bar{c}}$. 
	Then $d$ must not be in $Q_{\bar{a}b\bar{c}}$, so that $b$ is separated from $d$, 
	and $d$ must not be in $Q_{a\bar{b}\bar{c}}$, so that $a$ is separated from $d$. 
	But now there is no way to separate $d$ from $c$, a contradiction.
\end{proof}
  
  A vertex $u$ is \emph{tight} if its degree in $G-E_{i-1}$ is more than $\eps n - 2$ 
  if $d_{i-1}(u) = 0$, or more than $\sqrt{\eps'}n - 2d_{i-1}(u) - 2$ if $d_{i-1}(u) > 0$.
  An edge~$f$ is \emph{available for $P_i$} if $f \not\in E(\calC_{i-1}) \setminus E(\calC_0)$
  and the extremes of the edges in $E^f$ are not tight.
  
  To transform $Q_i$ into $P_i$, we will proceed as follows. 
  Start with~$P'_i$ being one of the paths in $Q_i$ and let $Q'_i = Q_i \setminus \{P'_i\}$.
  While~$Q'_i$ is non-empty, let $P$ be one of the paths in~$Q'_i$.
  Call~$y$ one of the ends of~$P$ and $x$ one of the ends of $P'_i$.
  An $(x,y)$-path in~$G$ is \emph{good} if it has length at most $L$, 
  all of its edges are available and its inner vertices are not in~$V(P_i) \cup V(Q'_i)$.
  If a good $(x,y)$-path exists, we extend $P'_i$ by gluing $P'_i$ and~$P$; 
  we remove~$P$ from $Q'_i$, and repeat this process until $Q'_i$ is empty.
  When $Q'_i$ is empty, we let $P_i$ be $P'_i$. Recall that
  $\calC_i = (\calC_{i-1} \setminus \{Q_i\}) \cup \{P_i\}$, and we let $E_i$ 
  be~$E_{i-1}$ minus all edges contained in more than three 2-matchings of~$\calC_i$ 
  and all edges not separated by~$\calC_i$ from some other edge of~$G$.
  
  This process is well-defined if the required $(x,y)$-good path exists at every point in the construction.
  We will show that, indeed, assuming that the invariants hold, there is 
  always a good $(x,y)$-path to be chosen in the gluing process above. 
  Then, to complete the proof, we will prove that the invariants hold even 
  after $Q_i$ is modified by the choice of any good path.
  
  First, note that the number of vertices in $P'_i$ not in $Q_i$ is less than~$L \eps'n$.
  Indeed, each connecting path has at most $L$ inner vertices and $Q_i$ is $\eps'$-compact, 
  hence $Q_i$ has no more than~$\eps'n$ paths.
  Thus we use less than~$\eps'n$ 
  connecting paths to get to $P_i$.  If~$\eps' < \delta /(4L)$, then 
  the number of vertices in $P'_i$ not in $Q_i$ is less than~$\delta n / 4$.
  
  Second, let us consider the tight vertices. 
  We start by arguing that $x$ is not tight. 
  This happens because $d_i(x) = d_{i-1}(x) - 1$ and, by invariant~\ref{inv-Delta}, 
  the degree of $x$ in~$G-E_{i-1}$ is at most $\sqrt{\eps'}n - 2d_{i-1}(x) = \sqrt{\eps'}n - 2d_i(x) - 2$. 
  For the same reasons, $y$ is not tight. 
  Now, note that $E_{i} \setminus E_{i-1} \subseteq \bigcup \{E^f : f \in E(P_i) \setminus E(Q_i)\}$.  
  Hence, $|E_{i} \setminus E_{i-1}| \leq 4L\eps'n$ by Claim~\ref{clm:atmost3}
  and because $Q_i$ consists of at most $\eps'n$ paths. 
  This, $\Delta(G - E(\calQ)) \leq \eps' n$, and $d_i(G) \leq \eps' n$ imply that the maximum number of tight vertices is 
  at most $(\eps'n + 4L\eps'\beta n)/(\sqrt{\eps'}-2\eps') = \eps'(1 + 4L\beta)n/(\sqrt{\eps'}-2\eps')$.
  As long as~$2\eps' < \sqrt{\eps'}/2$, that is, $\eps' < 1/16$, 
  we have that this number is less than $2\sqrt{\eps'}(1 + 4L\beta)n$.
  If additionally $\eps' < (\delta/(8(1+4L\beta)))^2$, we have that
  the number of tight vertices is less than~$2\sqrt{\eps'}(1+4L\beta)n < \delta n/4$. 
  
  Third, $|E(\calC_{i-1}) \setminus E(\calC_0)| < 4L\eps'n(i-1) < 4L\eps'\beta n^2$
  because $i \leq \beta n$.  Hence, by invariant~\ref{inv-3x}, at most 
  $4L\eps'\beta n^2$ edges are used more than three times by $\calC_{i-1}$.  
  Let $e \in E(\calC_{i-1}) \setminus E(\calC_0)$. 
  Because~$Q_i$ is~$(\delta,L)$-robustly-connected, there exist $\ell \leq L$ 
  and $\delta n^\ell$ $(x,y)$-paths in~$G$, each with~$\ell$ internal vertices, 
  all in $V(G) \setminus V(Q_i)$.
  If $e$ is not incident to $x$ or $y$, then the number of $(x,y)$-paths 
  in~$G$ with~$\ell$ internal vertices and containing $e$ is at most $n^{\ell-2}$. 
  Hence, the number of $(x,y)$-paths in~$G$ with~$\ell$ internal vertices, 
  containing an edge in $E(\calC_{i-1}) \setminus E(\calC_0)$ not incident 
  to~$x$ or~$y$, is less than $4L\eps'\beta n^\ell$. 
  If $e$ is incident to~$x$ or~$y$, then the number of $(x,y)$-paths in~$G$ 
  with~$\ell$ internal vertices and containing $e$ is at most $n^{\ell-1}$.
  But, there are less than $\sqrt{\eps'} n$ edges incident to~$x$ and less 
  than $\sqrt{\eps'} n$ edges incident to~$y$ contained in more than three 
  2-matchings in $\calC_{i-1}$, by invariant~\ref{inv-Delta}.  
  Thus, the number of $(x,y)$-paths in~$G$ of length~$\ell$ containing an edge 
  in $E(\calC_{i-1}) \setminus E(\calC_0)$ incident to $x$ or $y$ is less 
  than $2\sqrt{\eps'}n^\ell$.  We can choose $\eps'$ small enough so that 
  $4L\eps'\beta+2\sqrt{\eps'} < \delta/4$, and thus at most $\delta n^\ell/4$ 
  $(x,y)$-paths of length $\ell$ contain some edge of $E(\calC_{i-1}) \setminus E(\calC_0)$.
  
  Summarising, we have concluded that, for $\eps'$ small enough, the number of vertices 
  in~$P'_i$ not in $Q_i$ is less than $\delta n/4 \leq \delta n^\ell/4$, the number of 
  tight vertices is also less than $\delta n/4 \leq \delta n^\ell/4$, and the number 
  of $(x,y)$-paths containing some edge in~$E(\calC_{i-1}) \setminus E(\calC_0)$ 
  is less than $\delta n^\ell/4$.  This means that at least $\delta n^\ell/4$ of the 
  $\delta n^\ell$ $(x,y)$-paths in~$G$, each with~$\ell$ internal vertices, 
  all in $V(G) \setminus V(Q_i)$, are good.  As long as $n_0$ is such that 
  $\delta n_0^\ell/4 \geq \delta n_0 / 4 \geq 1$, there is a good $(x,y)$-path. 
  
  Now let us verify the invariants.
  By the definition of $E_i$, invariant~\ref{inv-sep} holds for $i$ 
  because $\calC_i$ separates every edge in~$E_i$ from all other edges of~$G$.
  Invariant~\ref{inv-3x} holds because edges in more than three 2-matchings in $\calC_i$ lie 
  in used connecting paths, that is, lie in $E(P_j) \setminus E(Q_j)$ for some $j$ with $1 \leq j \leq i$.
  For invariant~\ref{inv-Delta}, observe that $E_{i} \setminus E_{i-1} \subseteq E(P_i)$, 
  so the degree of $v$ from $G-E_{i-1}$ to $G-E_i$ decreases only for untight vertices, 
  and by at most two.  As the degree of an untight vertex $u$ in $G-E_{i-1}$ is at 
  most~$\eps n - 2$ if $d_{i-1}(u) = 0$ and at most~$\sqrt{\eps'}n - 2d_{i-1}(u) - 2$ 
  if $d_{i-1}(u) > 0$, every vertex~$u$ in~$G-E_i$ has degree at most~$\eps n$ if~$d_i(u) = 0$
  and at most $\sqrt{\eps'}n - 2d_i(u)$ if $d_i(u) > 0$, also because $d_i(u) \leq d_{i-1}(u)$.
  So invariant~\ref{inv-Delta} holds.
\end{proof}

\section{Separating the last few edges} \label{sec:separating}

In this section we deal with a subgraph $H$ of $G$, of small maximum degree, 
whose edges are not separated by the path family obtained in the previous sections.
This is done in~\Cref{lemma:lastfewpaths} but first we need some auxiliary results.
The first step of the proof is to find a family of matchings which separates the edges of $H$.

\begin{lemma} \label{lemma:separatingmatchings}
Let $\Delta \geq 0$ and let $H$ be an $n$-vertex
graph with $\Delta(H) \leq \Delta$. Then there is a collection
of $t \leq 300 \sqrt{\Delta n}$ matchings $M_1, \dotsc, M_t \subseteq H$ such that
\begin{enumerate}[\upshape{\textbf{(M\arabic*)}},labelindent=0pt,labelwidth=\widthof{\ref{item:Mlast}},leftmargin=!]
	\item each edge in $H$ belongs to exactly two matchings $M_i, M_j$; and\label{item:Mtwo}
	\item for each $1 \leq i < j \leq t$, the matchings $M_i, M_j$ have
	at most one edge in common.\label{item:Mlast}
\end{enumerate}
\end{lemma}

We also need the asymmetric version of the Lovász Local Lemma
(cf.~\cite[Theorem~1.1]{Spencer1977}).

\begin{theorem}[Asymmetric Lovász Local Lemma] \label{theorem:ALLL}
  Let $\mathcal{E} = \{A_1, \dotsc, A_n\}$ be a collection of events
  such that each $A_i$ is mutually independent of $\mathcal{E} -
  (\mathcal{D}_i \cup A_i)$, for some $\mathcal{D}_i \subseteq
  \mathcal{E}$.  Let $0 < x_1, \dotsc, x_n < 1$ be real numbers such
  that, for each $i \in \{1, \dotsc, n\}$,
  \begin{equation}
    \probability[A_i] \leq x_i \prod_{A_j \in \mathcal{D}_i}(1 - x_j).
    \label{equation:lll}
  \end{equation}
  Then $\probability\left( \bigcap_{i=1}^n \overline{A_i} \right) \geq
  \prod_{i=1}^n (1 - x_i) > 0$.
\end{theorem}

\begin{proof}[Proof of \cref{lemma:separatingmatchings}]
  Let $D = 256 \sqrt{\Delta n}$.  Let $M = \{ 1, \dotsc, D+1 \}$ and
  let $M^{(2)}$ consist of all subsets of size two of $M$.  We define
  a function $\phi: E(H) \rightarrow M^{(2)}$ by choosing $\phi(e) \in
  M^{(2)}$ uniformly at random for each $e \in E(H)$.  We will show
  that, with positive probability,
  \begin{enumerate}
  	\item \label{item:lll1} $\phi$ is injective, and
  	\item \label{item:lll2} for each vertex $v \in V(H)$, the sets
  	$\phi(vw)$ for $w \in N(v)$ are pairwise disjoint.
  \end{enumerate}
  We define a sequence of ``bad'' events to use \Cref{theorem:ALLL}.
  For distinct $e, f \in E(H)$, let $\mathcal{A}_{e,f}$ be the event
  that $\phi(e) = \phi(f)$.  For each pair of adjacent edges $e, f \in
  E(H)$, let $\mathcal{B}_{e,f}$ be the event that $\phi(e) \cap
  \phi(f) \neq \emptyset$.  Thus \ref{item:lll1}--\ref{item:lll2} hold
  if we avoid all $\mathcal{A}_{e,f}$ and $\mathcal{B}_{e,f}$.
  
  Note first that, for each $e, f$, we have 
  \begin{align*}
  	\probability[\mathcal{A}_{e,f}]
  	& = \binom{D+1}{2}^{-1} = \frac{2}{D(D+1)} \leq \frac{2}{D^2}, \\
  	\probability[\mathcal{B}_{e,f}]
  	& = (2D-1)\binom{D+1}{2}^{-1} = \frac{2(2D-1)}{D(D+1)} \leq \frac{4}{D}.
  \end{align*}
  
  Define $d_A := \Delta n$ and $d_B := 4 \Delta$.  Note that each
  event $\mathcal{A}_{e,f}$ or $\mathcal{B}_{e,f}$ is independent of
  all other events $\mathcal{A}_{e',f'}$ except if $\{e, f\} \cap
  \{e', f'\} \neq \emptyset$.  Given $\{e,f\}$, the number of such
  intersecting pairs $\{e',f'\}$ is at most $d_A$.  Similarly, each event
  $\mathcal{A}_{e,f}$ or $\mathcal{B}_{e,f}$ is independent of all
  but at most $d_B$ events of type $\mathcal{B}_{e',f'}$.
  
  For each event $\mathcal{A}_{e,f}$ define $x_{e,f} := x_A := d^{-1}_A$ and 
  for each event $\mathcal{B}_{e,f}$ define $y_{e,f} := x_B := d^{-1}_B$.
  We will show that the requirement \eqref{equation:lll} of the 
  Asymmetric Lovász Local Lemma is satisfied with these choices.
  
  Indeed, for an event of type $\mathcal{A}_{e,f}$, we use the fact
  that $1 - x \geq 2^{-2x}$ for $0 \leq x \leq 1/2$ to show that
  \begin{align*}
  	x_{A} \left(1 - x_A \right)^{d_A} \left(1 - x_B \right)^{d_B}
  	& \geq x_A 2^{-2x_A d_A} 2^{-2x_B d_B} = x_A 2^{-4} = \frac{1}{16 \Delta n} \geq \frac{2}{D^2}
  	\geq \probability[\mathcal{A}_{e,f}],
  \end{align*}
  and, for an event of type $\mathcal{B}_{e,f}$, we have
  \begin{align*}
  	x_{B} \left(1 - x_A \right)^{d_A} \left(1 - x_B \right)^{d_B}
  	& \geq x_B 2^{-2x_A d_A} 2^{-2x_B d_B} = x_B 2^{-4} = \frac{1}{64 \Delta} \geq \frac{4}{D}
  	\geq \probability[\mathcal{B}_{e,f}].
  \end{align*}
  Thus \Cref{theorem:ALLL} guarantees there is a function $\phi$
  satisfying \ref{item:lll1}--\ref{item:lll2}.  This function
  defines the matchings: for each $1 \leq i \leq D+1$ we let $M_i$
  consist of the edges $e \in E(H)$ such that $i \in \phi(e)$.
  Then $\phi(e) \in M^{(2)}$ ensures that each edge belongs to
  exactly two $M_i$'s, condition \ref{item:lll1} ensures that each
  pair of $M_i$, $M_j$ has at most one edge in common, and
  condition \ref{item:lll2} ensures that each $M_i$ is a matching.
  Since $D+1 \leq 300 \sqrt{\Delta n}$, we are done.
\end{proof}

Now we prove the main result of this section, which finds the required family of paths that separate $E(H)$.
The proof proceeds by using the matchings found in the previous lemma and covering those matchings with paths.

\begin{lemma} \label{lemma:lastfewpaths}
  Let $\eps, \delta, L > 0$ and let $G$ and $H$ be $n$-vertex graphs with $H \subseteq G$ 
  such that $\Delta(H) \leq \eps n$ and $G$ is $(\delta, L)$-robustly-connected. 
  Then there exist paths $\{P_i\}_{i=1}^r$, $\{Q_i\}_{i=1}^r$ in $G$, 
  with $r \leq 600L \delta^{-1} \sqrt{\eps} n$, 
  such that, for each $e \in E(H)$, there exist distinct $1 \leq i < j \leq r$ 
  such that $\{ e \} = E(P_i) \cap E(P_j) \cap E(Q_i) \cap E(Q_j)$.
\end{lemma}

\begin{proof}
  Apply \Cref{lemma:separatingmatchings} to $H$ (with $\eps n$ in place of $\Delta$), to obtain a collection
  of $t \leq 300 \sqrt{\eps} n$ matchings $M_1, \dotsc, M_t$ such that
  each edge in $H$ belongs to exactly two of these matchings; and
  each two distinct matchings have at most one edge in common.
  
  Separate the edges of each $M_i$ into $r_i \leq 2L \delta^{-1}$
  matchings $M_{i, 1}, \dotsc, M_{i, r_i}$ where each $M_{i, j}$ has
  less than $\delta n / (4 L)$ edges.  Let $r = \sum_i r_i$ be the total
  number of matchings obtained after doing this.  Since we have $t$
  matchings $M_i$ initially, after this process, we have obtained at
  most $r \leq t 2 L \delta^{-1} \leq 600 L \delta^{-1} \sqrt{\eps} n$
  matchings $M_{i, j}$.  We rename and enumerate the new matchings to
  be $M'_1, \dotsc, M'_r$ from now on.
  
  The next step is to obtain, for each $1 \leq i \leq r$, two paths
  $P_i$ and $Q_i$ of $G$ with the property that $E(P_i) \cap E(Q_i) = M'_i$. 
  For that, let $x_1 y_1, x_2 y_2, \dotsc, x_\ell y_\ell$ be the edges of~$M'_i$. 
  Since $G$ is $(\delta, L)$-robustly-connected, 
  there exist $1 \leq \ell \leq L$ and at least $\delta n^\ell$ many internally 
  vertex-disjoint $(x_1, x_2)$-paths with $\ell$ inner vertices each.  
  Because $|V(M'_i)| = 2 |M'_i| \leq \delta n / (2 L) < \delta n$, 
  there exists an $(x_1, x_2)$-path $P^{\smash{(1)}}_i$ of length at most $L$ 
  which is internally vertex-disjoint from $V(M'_i)$.  
  Similarly, we can find a $(y_1, y_2)$-path $Q^{\smash{(1)}}_i$ of length 
  at most $L$ which is internally disjoint from $V(M'_i) \cup V(P^{\smash{(1)}}_i)$.  
  We proceed in this fashion iteratively, finding for each $1 \leq k < \ell$, 
  in order, some $(x_k, x_{k+1})$-path $P^{\smash{(k)}}_i$ and a $(y_k, y_{k+1})$-path 
  $Q^{\smash{(k)}}_i$, both of length at most $L$, and both internally disjoint 
  from $V(M'_i)$ and from all previously chosen paths.  
  This can be achieved, because in each step the number of vertices we need 
  to avoid is at most $2 |M'_i| + 2 L |M'_i| \leq 4 L |M'_i| < \delta n$, 
  which implies that there is always one path available to choose.  
  We define $P_i$ as the path which starts with the edge $x_1y_1$, 
  then traverses the path $Q^{\smash{(1)}}_i$, then uses $y_2 x_2$, 
  then $P^{\smash{(2)}}_i$, etc., alternatingly using the paths $P^{\smash{(k)}}_i$ 
  and $Q^{\smash{(k)}}_i$, and covering all edges of $M'_i$.  
  We define $Q_i$ similarly, starting by the edge $y_1 x_1$, but then using the path $P^{\smash{(1)}}_i$, then $x_2 y_2$,
  then $Q^{\smash{(2)}}_i$, and so on.  
  Then $P_i, Q_i$ satisfy that $E(P_i) \cap E(Q_i) = M'_i$, as required.
  
  We define $\mathcal{P} = \{P_1, \dotsc, P_r\}$ and $\mathcal{Q} =
  \{Q_1, \dotsc, Q_r\}$.  By construction, each of them has the
  required number of paths.  Now we check that these families satisfy
  the required property.  Let $e \in E(H)$ be arbitrary.  By the
  choice of the matchings, there exist distinct $i_1, i_2$ such that
  $\{e\} = M_{i_1} \cap M_{i_2}$.  Suppose that $i, j$ are
  distinct such that $e \in M'_i \subseteq M_{i_1}$ and $e \in M'_j
  \subseteq M_{i_2}$.  It must happen that $\{e\} = M'_i \cap M'_j$.  
  Then, by the choice of $P_i, Q_i, P_j, Q_j$, we have that
  $M'_i = E(P_i) \cap E(Q_i)$ and $M'_j = E(P_j) \cap E(Q_j)$, and 
  therefore $E(P_i) \cap E(Q_i) \cap E(P_j) \cap E(Q_j) = M'_i \cap M'_j = \{e\}$, 
  as required.
\end{proof}

\section{Proof of the main result}\label{sec:mainresult}

Now we have the tools to prove our main result, from
which~\Cref{theorem:completegraph} and~\Cref{theorem:regular}
immediately follow (in combination with the lower bounds from
\Cref{proposition:lowerboundclique,proposition:lowerboundgeneral}).

\begin{theorem} \label{theorem:main}
	Let $\alpha, \rho, \varepsilon, \delta \in (0,1)$ and $L > 0$.
	Let $n$ be sufficiently large, and let $G$ be an $n$-vertex $(\alpha n \pm n^{1 - \rho})$-regular graph which is $(\delta,L)$-robustly-connected.
	Then $\ssp(G) \leq (\sqrt{3 \alpha + 1} - 1 + \varepsilon)n$.
\end{theorem}

\begin{proof}
  Let $\eps_2 := 1-1/(1 + \eps/2)$ and $\delta' = \eps_2^\ell \delta/2$.  
  Choose $\eps'$ and $n_0$ such that \Cref{lemma:connectingpaths} holds 
  with $(\eps \delta/(2400L))^2$, $L$ and $\delta'$ playing the roles of $\eps$, $L$ and $\delta$,
  respectively. From now on, we assume $n \geq n_0$ and let $\beta := \sqrt{3 \alpha + 1} - 1$.  
  Apply \Cref{lemma:firstpacking} to $G$ with $\eps_2$ and $\eps'$ playing the roles of $\eps$ and $\eps'$,
  respectively.  By doing this, we obtain a family $\mathcal{Q}$ of $2$-matchings which is 
  a $(\delta', L, (1+\eps/2)\beta, \eps')$-separator.  
  Thus $\mathcal{Q}$ consists of ${t := (1 + \eps/2)\beta n \leq (\sqrt{3 \alpha + 1} - 1 + \eps/2)n}$ 
  many $2$-matchings $Q_1, \dotsc, Q_t$, satisfying 
  \ref{item:firstpacking-connecting}--\ref{item:firstpacking-almostperfect}
  (with $\delta'$, $(1+\eps/2)\beta$ and $\eps'$ in place of $\delta$, $\beta$ and $\eps$).
	Next, we apply \Cref{lemma:connectingpaths} to $G$ and $\mathcal{Q}$.
	By the choice of $\eps'$ and $n_0$, we obtain an $(\eps \delta /(2400L))^2$-almost separating path system $\mathcal{P}$ in $G$ of size~$t$.
	
	Let $E' \subseteq E(G)$ be the subset of edges which are
        strongly separated by $\mathcal{P}$ from every other edge.
        Since $\mathcal{P}$ is $(\eps \delta/(2400L))^2$-almost
        separating, the subgraph $J := G - E'$ satisfies $\Delta(J)
        \leq (\eps \delta/(2400L))^2 n$.  By assumption, $G$ is
        $(\delta, L)$-robustly-connected, which allows us to apply
        \Cref{lemma:lastfewpaths} with $J$ and $(\eps
        \delta/(2400L))^2$ playing the roles of $H$ and $\eps$,
        respectively.  By doing so, we obtain two families
        $\mathcal{R}_1, \mathcal{R}_2$ of at most $\eps n / 4$ paths
        each, such that, for each $e \in E(J)$, there exist two paths
        $P_i, P_j \in \mathcal{R}_1$ and $ Q_i, Q_j \in \mathcal{R}_2$
        such that $\{e\} = E(P_i) \cap E(P_j) \cap E(Q_i) \cap
        E(Q_j)$.
	
	We let $\mathcal{P}' := \mathcal{P} \cup \mathcal{R}_1 \cup
        \mathcal{R}_2$.  Note that $\mathcal{P}'$ has at most $t +
        \eps n/2 \leq (\sqrt{3 \alpha + 1} - 1 + \eps)n$ many paths.  We
        claim that $\mathcal{P}'$ is a strong-separating path system
        for $G$.  Indeed, let $e, f$ be distinct edges in $E(G)$; we
        need to show that there exists a path in $\mathcal{P}'$ which
        contains $e$ and not $f$.  If $e \in E'$, then such a path is
        contained in $\mathcal{P}$, so we can assume that $e \in
        E(J)$.  There exist four paths $P_i, P_j, Q_i, Q_j \in
        \mathcal{P}'$ such that $\{e\} = E(P_i) \cap E(P_j) \cap
        E(Q_i) \cap E(Q_j)$, which in particular implies that one of
        these paths does not contain $f$.
\end{proof}

\section{Corollaries} \label{sec:corollaries}

Now we apply \Cref{theorem:main} to bound $\ssp(G)$ for graphs $G$
belonging to certain families of graphs.  In all cases, we just need
to check that the corresponding graphs are $(\delta,
L)$-robustly-connected for suitable parameters.

\begin{corollary}
	For each $\eps > 0$ and sufficiently large $n$, $\ssp(K_{n/2, n/2}) \leq (\sqrt{5/2} - 1 + \eps)n$.
\end{corollary}

\begin{proof}
  Let $\eps > 0$ be arbitrary and $n$ sufficiently large in terms
  of~$\eps$.  The graph $K_{n/2, n/2}$ is $n/2$-regular, so it is
  $\alpha n$-regular with $\alpha = 1/2$.  Pairs of vertices $x,y$ in
  the same part of the bipartition have $n/2$ neighbours in common;
  and pairs of vertices $x,y$ in different parts of the bipartition
  have $((n/2)-1)^2 \geq n^2/5$ many $(x,y)$-paths with two inner
  vertices each.  Hence, $K_{n/2, n/2}$ is $(1/5,
  2)$-robustly-connected.  By applying \Cref{theorem:main} with
  $\alpha = 1/2$, $\eps$, $\rho$ and $\delta = 1/5$ we obtain that
  $\ssp(K_n) \leq (\sqrt{5/2} - 1 + \eps)n$.
\end{proof}

%
%
%
%
%

Let us now describe a well-known family of graphs which satisfies the
connectivity assumptions of \Cref{theorem:main}.  Given $0 < \nu \leq
\tau \leq 1$, a graph $G$ on $n$ vertices, and a set $S \subseteq
V(G)$, the \emph{$\nu$-robust neighbourhood of $S$} is the set
$\operatorname{RN}_{\nu, G}(S) \subseteq V(G)$ of all vertices with at
least $\nu n$ neighbours in $S$.  We say that $G$ is a \emph{robust
$(\nu, \tau)$-expander} if, for every $S \subseteq V(G)$ with $\tau n
\leq |S| \leq (1 - \tau)n$, we have $|{\operatorname{RN}}_{\nu, G}(S)|
\geq |S| + \nu n$.  Many families of graphs are robust $(\nu,
\tau)$-expanders for suitable values of $\nu, \tau$, including large
graphs with $\delta(G) \geq dn$ for fixed $d > 1/2$, dense random graphs,
dense regular quasirandom graphs~\cite[Lemma 5.8]{KuhnOsthus2014},
etc.

\begin{corollary}\label{coro:robustexpander-robustconnected}
	For each $\eps, \alpha, \tau, \nu, \rho > 0$ with $\alpha \geq
        \tau + \nu$, there exists $n_0$ such that the following holds
        for each $n \geq n_0$.  Let $G$ be an $n$-vertex ${(\alpha n
          \pm n^{1 - \rho})}$-regular robust $(\nu, \tau)$-expander.
        Then $\ssp(G) \leq (\sqrt{3 \alpha + 1} - 1 + \eps)n$.
\end{corollary}


\begin{proof}
	By \Cref{theorem:main}, it is enough to prove that $G$ is
        $(\delta, L)$-robustly-connected with $\delta, L$ depending on
        $\nu$ only.  We will prove this holds with $L := \lceil
        \nu^{-1} \rceil$ and $\delta := (\nu / 4)^L 4^{-L^2}$.
	
	Let $x$ be any vertex, and let $N(x)$ be its neighbourhood.
        We define $R_0 = \emptyset$ and for each $i \geq 0$ we let
        $R_{i+1} = R_i \cup ( \operatorname{RN}_{\nu, G}(N(x) \cup
        R_i) \setminus N(x))$ if $|N(x) \cup R_i| \leq (1 - \tau)n$;
        or $R_{i+1} = V(G)$ otherwise.  By definition, $R_0 \subseteq
        R_1 \subseteq R_2 \subseteq \dotsb$.
	
	Since $G$ is a robust $(\nu, \tau)$-expander and $|N(x)| \geq
        \alpha n \geq \tau n$, it can be quickly checked that, for
        each $i \geq 0$ such that $|N(x) \cup R_i| \leq (1 - \tau)n$,
        the bound ${|R_{i+1} \setminus R_i| \geq \nu n}$ holds.  In
        particular, this implies that $R_L = V(G)$.  Indeed, suppose
        otherwise.  Then $R_L \neq V(G)$, therefore $|N(x) \cup R_{i}|
        \leq (1 - \tau)n$ for all $0 \leq i < L$, which implies that
        $|R_i \setminus R_{i-1}| \geq \nu n$ holds for each $1 \leq i
        \leq L$.  But then, since $L \geq \nu^{-1}$, we have
	\[ n > |R_L| = |R_L \setminus R_{L-1}| + \dotsb + |R_2 \setminus R_1| + |R_1| \geq L \nu n \geq n, \]
	a contradiction.
	
	Given $j \geq 1$, we let $T_j \subseteq V(G)$ be the set of
        vertices $v$ for which there are at least $(\nu n/4)^j
        4^{-j^2}$ many $(x,v)$-paths in $G$ with $j$ inner vertices
        each.  We claim that, for each $0 \leq i \leq L$, it holds
        that $R_i \subseteq T_1 \cup \dotsb \cup T_i$.  Before proving
        the claim we note that this is enough to conclude: as
        discussed before we have that $V(G) = R_L \subseteq T_1 \cup
        \dotsb \cup T_L$, so for each vertex $y \in V(G)$ there would
        exist ${1 \leq \ell \leq L}$ such that $y \in T_\ell$.  This
        implies that there exist at least $(\nu n/4)^\ell 4^{-\ell^2}
        \geq \delta n^\ell$ many $(x,y)$-paths with $\ell$ inner
        vertices each, as required.
	
	Now we prove the claim by induction on $i$, where the base
        case $i=0$ holds vacuously.  Assuming the claim for some $i <
        L$, we prove it for $i+1$.  Let $y \in R_{i+1}$ be arbitrary,
        it is enough to check that $y \in T_1 \cup \dotsb \cup
        T_{i+1}$.  By the inductive hypothesis, we can assume that $y
        \in R_{i+1} \setminus R_i$.  Note that $y$ has at least $\nu
        n$ neighbours in $N(x) \cup R_i$.  Indeed, if $|N(x) \cup R_i|
        \leq (1 - \tau)n$ then $y \in R_{i+1} \setminus R_i \subseteq
        \operatorname{RN}_{\nu, G}(N(x) \cup R_i)$ so indeed $y$ must
        have at least $\nu n$ neighbours in $N(x) \cup R_i$.
        Otherwise, if $|N(x) \cup R_i| > (1 - \tau)n$ then, since $y$
        has at least $\alpha n \geq (\tau + \nu) n$ neighbours, at
        least $\nu n$ of them must be in $N(x) \cup R_i$.
	
	We are done if $y$ has at least $\nu n/2$ neighbours in
        $N(x)$, because that immediately implies that $y \in T_1$.  We
        assume from now on that $|N(y) \cap N(x)| < \nu n/2$ and
        therefore $|N(y) \cap R_i| \geq \nu n /2$.  By the induction
        hypothesis, $R_i \subseteq T_1 \cup \dotsb \cup T_i$.  Observe
        that there must exist $1 \leq r \leq i$ such that $|N(x) \cap
        T_r| \geq \nu n / 2^{r+1}$, as otherwise we would have $|N(x)
        \cap R_i| < (\nu n/2) \sum_{r\geq 1} 2^{-r} \leq \nu n / 2$, a
        contradiction.  Fix such an $r$ from now on, and we will
        conclude by showing that $y \in T_{r+1}$.
	
	Indeed, for each $z \in N(y) \cap T_r$ there is a family
        $\mathcal{P}_z$ of at least $(\nu n/4)^r 4^{-r^2}$ many $(x,
        z)$-paths with $r$ inner vertices each.  We wish to extend the
        paths in $\mathcal{P}_z$ by including $y$ to obtain
        $(x,y)$-paths with $r+1$ inner vertices each.  This can only
        fail for some $P \in \mathcal{P}_z$ if $y \in V(P)$, but that
        can happen only for at most $rn^{r-1}$ paths.
        Since $|\mathcal{P}_z| \geq (\nu n/4)^r 4^{-r^2}$, using that $r \leq L$ and $1/n \ll \nu$ we can deduce that $|\mathcal{P}_z|/2 \geq rn^{r-1}$.
        This allows us to conclude that there are at least $|\mathcal{P}_z| / 2$ many
        $(x,y)$-paths with $r+1$ inner vertices which end with $zy$.
        By counting the paths for each choice of $z \in N(y) \cap
        T_r$, the number of desired $(x,y)$-paths is at least
	\begin{align*}
		\sum_{z \in N(y) \cap T_r} \frac{|\mathcal{P}_z|}{2} 
		& \geq \frac{|N(y) \cap T_r|}{2} \left( \frac{\nu n}{4} \right)^r 4^{-r^2} 
		 \geq \frac{\nu n}{4 \cdot 2^{r}} \left( \frac{\nu n}{4} \right)^r 4^{-r^2} 
		\geq \left( \frac{\nu n}{4} \right)^{r+1} 4^{-(r+1)^2},
	\end{align*}
	so $y \in T_{r+1}$, as claimed.
	This finishes the proof.
\end{proof}

%
%

\section{Conclusion} \label{section:conclusion}

To determine the maximum of $\wsp(G)$ and $\ssp(G)$ over all
$n$-vertex graphs $G$ remains an interesting problem. Falgas-Ravry,
Kittipassorn, Korándi, Letzter, and Narayanan\ (see \cite[Conjecture
1.2]{FKKLN2014} and the remarks afterwards) said that `it is not
inconceivable' that $\wsp(G) \leq (1 + o(1))n$ holds for all
$n$-vertex graphs $G$.  We have shown that $\ssp(G) \leq (1 + o(1))n$
holds for dense, regular, sufficiently connected $n$-vertex graphs.
Even the following could be
true\footnote{In~\cite[Theorem~10]{BCMP2016} it is stated that for
  each $\eps \in (0, 1/2)$ there exists some $n$ and an $n$-vertex
  graph $G$ such that $\ssp(G) \geq 2(1 - 2 \eps)n$, but unfortunately
  the proof has a flaw. The error in the proof appears
  in~\cite[Remark~10]{BCMP2016}, because the length of the longest
  path in $K_{\eps n, (1 - \eps)n}$ is $2 \eps n$ and not $\eps n +
  1$.}:

\begin{question}
	Does $\ssp(G) \leq (1 + o(1))n$ hold for \emph{all} $n$-vertex graphs $G$?
\end{question}

It would also be interesting to estimate $\wsp(G)$ and $\ssp(G)$ for
graphs not covered by our main result.  Complete bipartite graphs
$K_{a,b}$ with $a < b$ are an interesting open case.  It is also of
interest to weaken the conditions in our main result. For instance,
can the connectivity conditions in \Cref{theorem:main} be weakened?
Does $\Omega(n)$-vertex-connectivity suffice?

\subsection*{Acknowledgments}
This work began at the ChiPaGra workshop 2023 in Valparaíso, Chile,
funded by FAPESP-ANID Investigación Conjunta grant 2019/13364-7.
We thank the organizers and the attendees for a great atmosphere.
The third author thanks Matías Pavez-Signé for useful suggestions.

\sloppy\printbibliography

\end{document}